\documentclass[12pt,twoside,a4paper]{article}  
\usepackage{titling}
\usepackage{titletoc}
\usepackage{url}

\usepackage[utf8]{inputenc}
\usepackage[T1]{fontenc}
\usepackage{amsmath}
\usepackage{amsfonts}
\usepackage{amssymb}
\usepackage[numbers]{natbib}

\usepackage[bookmarksopen=false,breaklinks=true,%
      backref=page,pagebackref=true,plainpages=false,%
      hyperindex=true,pdfstartview=FitH,%
      pdfpagelabels=true,
      colorlinks=true,linkcolor=blue,%
      citecolor=red,urlcolor=red,%hypertexnames=false,
      colorlinks=true%
      ]%
   {hyperref}

\newtheorem{theorem}{Theorem}[subsection]

\newtheorem{proposition}[theorem]{Proposition}
\newtheorem{remark}[theorem]{Remark}
\newtheorem{remarks}[theorem]{Remarks}
\newtheorem{example}[theorem]{Example}

\newtheorem{corollary}[theorem]{Corollary}
\newtheorem{notation}[theorem]{Notation}
\newtheorem{definition}[theorem]{Definition}
\newtheorem{definota}[theorem]{Definition and notation}

\newtheorem{comput}{Computational Problem}
\newcounter{bidon}

\DeclareSymbolFont{lasy}{U}{lasy}{m}{n}
\SetSymbolFont{lasy}{bold}{U}{lasy}{b}{n}
\let\Box\undefined
\DeclareMathSymbol\Box{\mathord}{lasy}{"32}

\def \I{\mathbb{I}}
\def \J{\mathbb{J}}
\def \N{\mathbb{N}}
\def \Q{\mathbb{Q}}

\def \Z{\mathbb{Z}}

\def \S{\mathbf{S}}
\def \P{\mathbf{P}}
\def \V{\mathbf{V}}
\def \K{\mathbf{K}}

\def \R{\mathbf{R}}

\def \Inp{\textbf{Input: }}
\def \Output{\textbf{Output: }}

\def \Krc{\K^{\mathrm{rc}}}
\def \Vrc{\V^{\mathrm{rc}}}
\def \Prc{\P^{\mathrm{rc}}}
\def \Kac{\K^{\mathrm{ac}}}
\def \Vac{\V^{\mathrm{ac}}}

\def \KVP{(\K,\alb\V,\alb\P)}
\def \RVP{(\R,\alb\V_\R,\alb\P_\R)}
\def \KVPrc{(\Krc,\alb\Vrc,\alb\Prc)}

\def \ResK{\overline{\K}}
\def \ResKrc{\overline{\Krc}}

\def \G{\Gamma}
\def \GK{\G_{\K}}
\def \GKrc{\G_{\Krc}}

\def \GKd{\Gamma_\K^{dh}}

\def \GKi{\GK\cup\{+\infty\}}

\def \GKdi{{\GKd}\cup\{+\infty\}}

\def \cC{\mathcal{C}}
\def \U{\mathcal{U}}
\def \M{\mathcal{M}}
\def \MV{\M_{\V}}
\def \UV{\U_{\V}}

\def \sign{\mathrm{sign}}
\def \Sperv{\mathrm{Sperv}}

\def \Land{\bigwedge}

\def \qsl{$\Q$-semilinear }

\def \ovf{ordered valued field }
\def \ovfs{ordered valued fields }
\def \ovfz{ordered valued field}
\def \ovfsz{ordered valued fields}
\def \ovsf{ordered valued subfield }

\def \rcvf{real closed valued field }
\def \rcvfs{real closed valued fields }
\def \rcvfz{real closed valued field}
\def \rcvfsz{real closed valued fields}

\def \rcf{real closed  field }

\def \rcfsz{real closed  fields}

\def \pgn{Newton poly\-gon }

\def \NPAz{Newton Poly\-gon algorithm}

\def \vco{$(\leq,\preceq)$--cons\-truc\-tible }

\def \vci{$(<,\preceq)$--interval }
\def \vcis{$(<,\preceq)$--intervals }

\def \num{\textrm{n}$^{\mathrm{o}}$}

\def \xn {x_1,\ldots,x_n}
\def \Kx {\K[\xn]}

\def\eop{\hbox{}\nobreak\hfill
\vrule width 1.4mm height 1.4mm depth 0mm \par \smallskip}

\def\alb{\allowbreak}
\def\noi{\noindent}
\def\sni{\smallskip\noindent}
\def \mni{\medskip\noindent}
\def \ms{\medskip}

\def\gen#1{\left\langle{#1}\right\rangle}

\newenvironment{proof}[1]{
\trivlist \item[\hskip \labelsep{\bf #1}]}{\hfill\mbox{$\Box$}
\endtrivlist}

\begin{document}

\title{Generalized Taylor formulae, \goodbreak
computations in real closed valued fields\goodbreak
and quantifier elimination
}
\author{ Mari-Emi Alonso
\and
Henri Lombardi}

\date{2002}

\maketitle

This paper appeared as

{{Alonso, Mari-Emi and Lombardi, Henri.
 {Generalized {T}aylor formulae, computations in real closed
              valued fields and quantifier elimination},
p.\ 33-57 in the book {\it Valuation theory and its applications, {V}ol.\ {I}
              ({S}askatoon, {SK}, 1999)},
{Fields Inst.\ Commun.\ Vol.\ {32},}
 {Amer.\ Math.\ Soc., Providence, RI},
{(2002)},
}

%---------------------  abstract  ----------------------
\begin{abstract}
We use generalized Taylor formulae in order to give some simple
constructions in the real closure of an \ovfz. We deduce a new, simple
quantifier elimination algorithm for \rcvfs and some theorems about
constructible subsets of real valuative affine space.
\end{abstract}
%-----------------------------------



\medskip  \noindent
Mari-Emi Alonso. Universitad Complutense, Madrid, Espa\~na.
\\ Partially supported by: PB95/0563-A.
{\tt mariemi@mat.ucm.es} \\
Henri Lombardi. Laboratoire de Math\'ematiques, UMR CNRS 6623.
Universit\'e de Franche-Comt\'e, France.
{\tt henri.lombardi@univ-fcomte.fr}

 \newpage

%--- SECTION{sec Introduction}-------------------------------
\section*{Introduction} \label{sec Introduction}
\addcontentsline{toc}{section}{Introduction}
%------------------------------------------------

In this work, we consider the real closure of an ordered valued field
and search for simple computations giving a constructive content to
this real closure. We don't try to give sophisticated algorithms
which would allow better complexity.

\medskip
We consider an \ovf $\KVP$ with $\V$ its valuation ring and
  $\P$ its positive cone. Recall that this means that the following
properties hold
%--------------------begin array---------------
$$\begin{array}{lll}
\V+\V\subseteq \V,&  \V\times \V\subseteq \V,
& \exists x\in \K\setminus \V,
\\
&   &  \forall x,y\in \K \; (xy=1\; \Rightarrow \; x\in \V \lor
y\in \V),   \\
\P\times\P\subseteq \P,&  \P+ \P\subseteq \P,
& \exists x\in \K\setminus \P,
\\
&   & \forall x,y\in \K \; (x+y=0\; \Rightarrow \; x\in \P \lor
y\in \P),
\end{array}$$
%---------------------end array--------------
$$\qquad \forall
x,y\in\K\; [(x+y\in\V,x\in\P,y\in\P)\;\Rightarrow\;x\in\V].
$$
For $a,b$ in $\K$, we write $a \le b$ if and only if $b-a$ is in $\P$.
We shall use freely in the sequel some well known features of \ovfsz:
$\Q\subseteq \V$,
elements of  $\V$ bounded from below by some positive rational
are units in $\V$, and the non-units in $\V$ are the infinitesimal elements of~$\K$.

\smallskip Let $\S$ be a subring of $\V$ such that $\K$ is the
fraction field of $\S$.
We assume that $\S$ is an explicit ordered ring and that
divisibility inside $\V$ is testable for two arbitrary elements of
$\S$.
These are our minimal assumptions of computability.
If we want more assumptions in certain cases we shall make them explicit.

\smallskip We denote the real closure of $(\K,\P)$ by $(\Krc,\Prc)$,
and we write $\Vrc$ for the convex hull of $\V$ inside $\Krc$;
then $\Vrc$ is the unique order-compatible valuation ring extending $\V$.
We call $\KVPrc$  the real closure of $\KVP$.

\smallskip In sections \ref{sec Basic material} and  \ref{sec OVF} our
general purpose is to discuss computational problems
in $\KVPrc$ under our computability assumptions on
$\KVP$.

Each computational problem we shall consider has as
\textit{input} a
finite family $(c_i)_{i=1,\ldots,n}$ of
\textit{parameters} in the ring $\S$.
We call them the \textit{coefficients\/} of our computational problem.
Our algorithms with the previous minimal computability assumptions
work uniformly.
This means that some computations are made that give polynomials in
$\Z[C_1,\ldots,C_n]$,
and that all our tests are of the two following types:
%-----------------begin $$------------------
$$ \mathrm{Is}\; P(c_1,\ldots,c_n) \geq 0?   \qquad
\mathrm{ Does}\;\;Q(c_1,\ldots,c_n)\;\;\mathrm{ divide} \;
\;P(c_1,\ldots,c_n)\;\;\mathrm{in}\;\V ?
$$
%-----------------end $$------------------
We are not interested in how the answers to these tests are found.
We may imagine these answers given either by some oracles or by some
algorithms.

\medskip
Let us state precisely some other notations.
We shall denote the unit group of $\V$ by
$\UV$, and $\MV ={\V} \setminus \UV$  will be the maximal ideal.

We shall denote the residue field $\V/\MV$ of $(\K,\V)$ by
$\ResK$, and the value group,
$\K^{{\times}}/\UV$, by~$\GK$.
We use freely the value group's usual additive notation as well as its usual
group-ordering (also denoted by $\le$).
Recall that $\GKrc$ is the  divisible hull  $\GKd$  of $\GK$.
For $x\in\K$ we write $v(x)$ or $v_\K(x)$ the valuation of $x$ in
$\GKi$. So
%-----------------begin $$----------------
$$ v(0)=+\infty ,\; v(xy)=v(x)+v(y),\;
(x\geq 0,y\geq 0)\Rightarrow v(x+y)=\min(v(x),v(y)),
$$
%-----------------end $$------------------
and
%-----------------begin $$----------------
$$ \forall x\in\K\;\; (
(v(x)\geq0\;\Leftrightarrow\; x\in\V) \;
\land\;
(v(x)> 0\;\Leftrightarrow\; x\in\MV)
).
$$
%-----------------end $$------------------

We write $\Kac=\Krc[\sqrt{-1}]$ and we denote by $\Vac$ the natural valuation
ring of $\Kac$ extending~$\Vrc$:
for $a,b\in\Krc$, $v(a+b\sqrt{-1})=(1/2)v(a^2+b^2)$.

\smallskip In fact elements of $\GKdi$ are always defined through
elements of
$\S$ in the following form.
We say that the valuation of some element $x$ belonging to
  $\Kac$ is \emph{well determined} if we know
integers $m$ and $n$, elements $c_1,...,c_n$ in $\S$, and
two elements $F$ and $G$ of $\Z[C_1,...,C_n]$, such
  that,
setting $f=F(c_1,\ldots,c_n)$ ($f\neq 0$) and $g=G(c_1,\ldots,c_n)$,
there exists a unit $u$ in $\Vac$  with:
%-----------------begin $$----------------
$$f x^m = u g
$$
%-----------------end $$------------------
(in particular, $v(0)$ is well determined).

We read the previous formula as:
%-----------------begin $$----------------
$$ m\, v_{\Kac}(x) = v_\K(g)-v_\K(f),
$$
%-----------------end $$------------------
or more simply as:
%-----------------begin $$----------------
$$ m\, v(x) = v(g)-v(f).
$$
%-----------------end $$------------------
For $x,y\in \Kac$, we shall use the notation $x\preceq y$ for $v(x)\leq v(y)$
(i.e., $y=x=0\; \lor\; (x\neq 0\land y/x\in\Vac)$).
%-- Example{val group}-----------------

\mni
\textbf{Example.}
Let us explain the computations that are necessary
  to compare $3v(x_1)+2v(x_2)$ to $7v(x_3)$ when the valuations
are given by
%-----------------begin $$----------------
$$f_1 x_1^{m_1} = u_1 g_1,\; f_2 x_2^{m_2} = u_2 g_2,\; f_3 x_3^{m_3}
= u_3 g_3\quad (g_1,g_2,g_3\neq0).
$$
%-----------------end $$------------------
We consider the LCM $m=m_1n_1=m_2n_2=m_3n_3$ of $m_1,m_2,m_3$.
We have
$$f_1^{n_1} x_1^{m} = u_1^{n_1} g_1^{n_1},\; f_2^{n_2} x_2^{m} =
u_2^{n_2} g_2^{n_2},\; f_3^{n_3} x_3^{m} = u_3^{n_3} g_3^{n_3}.
$$
So  $3v(x_1)+2v(x_2)\leq 7v(x_3)$ iff
$g_1^{3n_1}g_2^{2n_2}f_3^{7n_3}\preceq
f_1^{3n_1}f_2^{2n_2}g_3^{7n_3}$.

\ms The reader can easily verify that computations we shall run
in the value
group are always meaningful under our computability
assumptions on the ring $\S$.

In the same way, elements of the residue field will in general be
defined from elements of $\V$. So computations inside the residue
field are given by computations inside $\S$.

\medskip
We now give an outline of the paper.

In section \ref{sec Basic material} we give some basic tools used in 
the rest of
the paper. First we recall the Newton Polygon Algorithm and the Generalized
Tschirnhaus Transformation.
Then we insist on Generalized Taylor formulae, which
are formulae giving
$P(x)$ on a Thom interval as a sum of terms all having the same sign.
This feature
allows us to give a good description for $v(P(x))$ with the crucial Theorem
\ref{thValThom}.
This allows us to give a nice description for ``constructible"
subsets of the real line in the context of \rcvfs (cf.\  Theorem
\ref{thVcoLine}).

In section \ref{sec OVF} we settle three basic computational problems
  in the real closure of an \ovfz.
We solve the first problem by a simple trick
(subsection \ref{subsec.Solve1}).
The consequence is that
when we know how to compute in a given \ovfz, we know how to compute in its
real closure.
This can be seen as a not too difficult extension of basic algorithms
in \rcfsz.
Solving the second problem is possible by using our first algorithm,
but we prefer to develop another algorithm, similar to the Cohen-H\"ormander
algorithm for ordered fields.
We get in this way nice uniform results describing
precisely some generalizations of the complete tableau of signs
in the real closed
case (Theorems \ref{propCHV} and \ref{propMCHV}).

In section \ref{sec qela} we give parametrized
versions of previous algorithms
(Theorems  \ref{propCHV2} and~\ref{propMCHV2}),
and we apply these results to
quantifier elimination in \rcvfsz.
We consider the first order theory of
\rcvfs based on the language of ordered fields
$(0,1,+,-,\times,=,\leq)$ to which we add the
predicate $x\preceq y$.
So, all constants and variables represent elements in $\K$ (this corresponds to
our previously explained computability assumptions). We get the following
theorem.

%-- Thm{thCHV}-------

\ms
\textbf{Theorem \ref{thCHV}}
  \textit{Let $\Phi(\underline{a},\underline{x}) $
be a quantifier free formula in the first order theory of
\rcvfsz. We view the $a_i$'s as parameters and the $x_j$'s as variables.}

\textit{Then one can give a quantifier free formula $\Psi(\underline{a})$
such that the two formulae
$\;\exists \underline{x}\;\Phi(\underline{a},\underline{x})\;$ and
$\;\Psi(\underline{a})\; $
are equivalent in the formal theory.
(The terms appearing in the formulae $\Phi$ and $\Psi$ are
$\Z$--polynomials in the parameters, and, in the case of
$\Phi$, also in the variables.)
  }

\ms
%--- end-Thm---------------------------
We think we have given here a rather simple proof of this fundamental, well
known
result (see e.g., \cite{CD}).

We also get the following abstract form of the previous theorem.
%-- Thm{thSperv}-------

\ms
\textbf{Theorem \ref{thSperv}}
  \textit{Let us denote the real-valuative spectrum of a commutative ring $A$ by
$\Sperv A$. Then the canonical mapping from $\Sperv A[X]$ to $\Sperv A$
transforms any constructible subset into a constructible subset.
  }

\ms
%--- end-Thm---------------------------

In section \ref{secConstSubSpace} we apply the parametrized algorithms in order
to
study constructible subsets (in the meaning of \rcvfsz).
First, we get the analogue
of the Tarski-Seidenberg principle.

%-- Thm{thConsProj}-------

\ms
\textbf{Theorem \ref{thConsProj}}
\textit{Let $\KVP$ be an \ovsf of a \rcvf $\RVP$. Let $\pi$
be the canonical projection from $\R^{n+r}$ onto  $\R^{n}$.
Let $S\subseteq \R^{n+r}$ be any \vco set  defined over $\KVP$.
Assume that the sign test and the divisibility test are explicit
inside the ring generated by the coefficients of the polynomials
that appear in the definition of $S$.
Then a description of the projection $\pi(S)\subseteq\R^{n}$
can be computed in a uniform way by an algorithm that uses only
rational computations, sign tests and divisibility tests.}

\textit{In particular, the complexity of a description of $\pi(S)$
is explicitly bounded in terms of the complexity of a description of $S$.
  }

\ms
%--- end-Thm---------------------------

Finally we construct a kind of stratification for
\vco sets, that we call stratification \`a la Cohen Hormander because it is a
further development of the same notion for semialgebraic sets (cf. \cite{BCR}
chapter~9),
  and we finish the paper with the
following cell-decomposition theorem (for a precise definition of 
\qsl functions
see definition \ref{defQsl}).

%-- Thm{Cell decomposition theorem}----

\ms
\textbf{Theorem \ref{thCellDec1}} \textrm{(Cell decomposition theorem)}
\textit{
Let $\KVP$ be an \ovsf of a \rcvf $\RVP$.
Let $g_1,\ldots,g_s$ be nonzero polynomials in $\Kx$.
Consider a linear change of variables together with
  a family $(f_{i,j})_{i=1,\ldots,n;j=1,\ldots,\ell_i}$  that give a
stratification
for $(g_1,\ldots,g_s)$. Assume that this stratification is constructed \`a la
Cohen-H\"ormander. Consider any $k$-dimensional stratum $C_\varepsilon$
corresponding to this stratification.
Then there is a Nash isomorphism
$$h:(\R^+)^k\longrightarrow C_\varepsilon,\;\;\;
(t_1,\ldots,t_k)\longmapsto h(t_1,\ldots,t_k)$$
with the following property.}

\textit{If  $S$ is any \vco subset described from $g_1,\ldots,g_s$, then
  $S\, \cap\, C_\varepsilon$ is a finite union of cells $h(L_i)$, where each
$L_i$
can be defined as
%-----------------begin $$----------------
$$\left\{(t_1,\ldots,t_k)\in(\R^+)^k\; :\;
\Land\nolimits_\ell a_\ell(\tau)=\alpha_\ell
\;\land\; \Land\nolimits_m  b_m(\tau)>\beta_m\right\},
$$
%-----------------end $$------------------
where $\tau=(\tau_1,\ldots,\tau_k)=(v(t_1),\ldots,v(t_k)$,
the $a_\ell$'s and $b_m$'s
are $\Z$-linear forms w.r.t.\  $\tau$,
and  $\alpha_\ell,\;\beta_m\in \GKd$.}

\textit{Moreover, each $\tau_i$ is a \qsl function in some
$v(F_j(x_1,\ldots,x_n))$'s (with $F_j$ explicitly computable in $ \Kx$).
  }

%\medskip
%--- end-Thm---------------------------

%--- SECTION{Basic material}---------------------------------
\section{Basic material} \label{sec Basic material}
%---------------------------------------

%--- SUBsection{Newton Polygon} ------------------------
\subsection{The Newton Polygon}
  \label{subsec NP}
%--------------------------------------------

Here we recall the well known \NPAz.

A \textit{multiset } is a set with (nonnegative) multiplicities, or
equivalently
a list defined up to permutation. E.g., the roots of a polynomial
$P(X)$ repeated according to multiplicities form a multiset in the
algebraic closure of the base field. We shall use the notation
$[x_1,\ldots,x_d]$ for the multiset corresponding to the list
$(x_1,\ldots,x_d)$. The \textit{cardinality} of a multiset is the length
of a corresponding list, i.e., the sum of multiplicities occurring in
the multiset.

\smallskip The \pgn of a polynomial $P(X)=\sum_{i=0,\ldots,d}
p_iX^i\in \K[X]$
(where $p_d\not= 0$) is obtained from the list of pairs in
$\N\times (\GKi)$
$$((0,v(p_0)),(1,v(p_1)),\ldots,(d,v(p_d))).
$$
The \pgn is ``the bottom convex hull" of this list.
It can be formally defined as the extracted list
$\,((0,v(p_0)),\ldots,(d,v(p_d)))\,$ verifying:
two pairs $(i,v(p_i))$ and $(j,v(p_j))$ are two consecutive vertices
of the \pgn iff:

\centerline {if  $0\leq k< i~$ then
$~(v(p_j)-v(p_i))/(j-i)~>~(v(p_i)-v(p_k)/(i-k))$}

\centerline {if $ i<k< j~$ then
$~(v(p_k)-v(p_i))/(k-i)~\ge~(v(p_j)-v(p_i))/(j-i)$}

\centerline {if $j<k\le d~$ then
$~(v(p_k)-v(p_j))/(k-j)~>~(v(p_j)-v(p_i))/(j-i)$}

\noi It is easily shown that if  $(i,v(p_i))$ and $(j,v(p_j))$
are two consecutive vertices in the \pgn of the polynomial $P$,
then the zeroes of $P$
in $\Kac$ whose valuation in   $\GKd$ equals
$~(v(p_i)-v(p_j))/(j-i)~$ form a multiset with cardinality $j-i$.
%-- Computational problem{val group}---

\ms
\textbf{Computational problem 0}
\textit{
\emph{(Multiset of valuations of roots of polynomials)}  \\
\Inp Let $P\in \K [X]$ be a polynomial over a valued field
$(\K ,{\V})$. \\
\Output The multiset $[v(x_1),\ldots,v(x_n)]$ where $[x_1,\ldots,x_n]$
is the multiset of roots of $P$ in $\Kac$. }

\ms
%---end-comput-----------------------------------
%-- Algo{NP}-------------------------
\textbf{\NPAz}\\
The number $n_\infty$ of roots equal to 0 (i.e., with infinite
valuation) is read on $P$. Let $P_0:=P/X^{n_\infty}$. Compute the
Newton polygon of $P_0$, compute the slopes and output the answer.
\eop
%----end-algo---------------------------------------

%--- SUBsection{Tschirnhaus} ---------------------------
\subsection{Generalized Tschirnhaus transformation}
\label{subsec Tschirnhaus}

%---------------------------------------
We recall here the well known (generalized) Tschirnhaus transformation,
which we will use freely in our computations.

Let $\K$ be a field, $(P_j)_{j=1,\ldots,m}$ be a family of monic
polynomials in $\K [X]$, and
%-----------------begin $$------------------
$$P_j(X)=(X-x_{j,1})\times \cdots\times (X-x_{j,d_j})
$$
%-----------------end $$------------------
their decompositions in $\Kac$. Let  $Q(Y_1,\ldots,Y_m)$ be a
polynomial in  $\K[Y_1,\ldots,Y_m]$.
Then the polynomial
%-----------------begin $$------------------
$$T_Q(Z)=(Z-Q(x_{1,1},\ldots,x_{m,1}))\times \cdots\times
(Z-Q(x_{1,d_1},\ldots,x_{m,d_m}))
$$
%-----------------end $$------------------
is the characteristic polynomial of $A_Q$ where $A_Q$ is the matrix of
the multiplication by $Q(y_1,\ldots,y_m)$ inside
the $d$-dimensional $\K$-algebra
%-----------------begin $$------------------
$$\K[y_1,\ldots,y_m]:=\K [Y_1,\ldots,Y_m]/\gen{P_1(Y_1),\ldots,P_m(Y_m)}
$$
%-----------------end $$------------------
($d=d_1\cdots d_m$, and $y_i$ is the class of $Y_i$ modulo
$\gen{P_1,\ldots,P_m}$).

Now let $R\in \K [Y_1,\ldots,Y_m]$ with $R(\underline{x})\not= 0$ for
all
$m$-tuples $\underline{x}=(x_{1,r_1},\ldots,\alb x_{m,r_m})$.
So $A_R$ is an invertible matrix. Let $F=Q/R$,  then the polynomial
%-----------------begin $$------------------
$$
T_F(Z)=(Z-F(x_{1,1},\ldots,x_{m,1}))\times \cdots\times
(Z-F(x_{1,d_1},\ldots,x_{m,d_m}))
$$
%-----------------end $$------------------
is the characteristic polynomial of $A_Q(A_R)^{-1}$.

%--- SUBsection{subsecGTF}------------------------------
\subsection{Generalized Taylor Formulas}
\label{subsecGTF}
%---------------------------------------
%---- paragraph {parausualTF}--------------

~

\ms\textbf{Using the usual Taylor formula for computing
valuations in $\GKd$.}
%---------------------------------------

For $P\in \K[X]$ we denote $P^{[k]}=P^{(k)}/k!$, where $P^{(k)}$
is the $k$-th derivative of $P$. Let $t=x-a$, and assume $\deg(P)=d$,
the usual Taylor formula at the point $a$ is
%-----------------begin $$----------------
$$ P(x) = P(a) + P^{[1]}(a)t + P^{[2]}(a)t^2 + \cdots
+P^{[d-1]}(a)t^{d-1} + P^{[d]}t^d.
$$
%-----------------end $$------------------
Now assume that $P^{[d]}>0$. Let $a_0$ be the greatest
real root of the product
$PP^{[1]}\cdots P^{[d-1]}$. If $a\ge a_0$ we see that all
$P^{[k]}(a)$ are $\ge0$ and we get the following expression for
the valuation  $v(P(x))$ when $x>a$
%-----------------begin $$----------------
$$ v(P(x))= \min(\nu_0,\nu_1+\tau,\nu_2+2\tau,\ldots,\nu_d+d\tau)
$$
%-----------------end $$------------------
where $\tau=v(t)$ and $\nu_j=v(P^{[j]}(a))$ (some $\nu_j$'s may be
infinite). So, w.r.t.\  the variable
  $\tau$ the valuation of $P(x)$ in $\GKd$ is piecewise linear and
increasing.
Note that $\tau$ decreases from $+\infty$
  to $-\infty$ when $t$ increases from $0$ to $+\infty$.

In the following paragraphs, we see that generalized
Taylor formulae allow us to give
a similar description of the valuation  $v(P(x))$
when $x$ is inside a Thom interval.
%---- paragraph {para GTF}-----------------

\ms\textbf{What are generalized Taylor formulae?}
%---------------------------------------

A fundamental example of algebraic evidence for a sign is given
by  generalized Taylor formulae,
which make explicit some consequences of Thom's
lemma in terms of algebraic identities.

Thom's lemma implies that the set of points
where a real polynomial and its
successive derivatives have fixed signs is an interval.
An easy proof, by
induction on the degree of the polynomial, is based on the
mean value theorem.
We can  translate this geometric fact under the
form of algebraic
  identities called {\em   Generalized Taylor Formulas} (GTF for
short).

\smallskip Let us see an example where $\deg(P) \le 4 $.

%-- Example{exa2}-------------------
\begin{example}
\label{exa2}
Consider the general polynomial of degree $4$
  $$P(X) = c_0 X^4 + c_1 X^3 + c_2 X^2 + c_3 X + c_4,$$
consider the following system of sign conditions for the polynomial
$P$
and its successive derivatives with respect to the variable $X$:
  $$H(X): P(X) > 0,\; P^{[1]}(X) < 0,\; P^{[2]}(X) < 0,\;
P^{[3]}(X) < 0,\; P^{[4]} > 0.$$
Consider also the system of sign conditions
  obtained by relaxing all the inequalities,
except one of them, e.g., the last one:
  $$H'(X): P(X) \ge 0,\; P^{[1]}(X) \le 0, P^{[2]}(X) \le 0,\;
P^{[3]}(X) \le 0,\; P^{[4]} > 0.$$
Thom's lemma implies that:
$$[\, H'(a), \;H'(b),\; a < x < b\,] \;\Longrightarrow\; H(x). $$
Put $e_1 = x - a$, $e_2 = b - x.$
Consider the following algebraic identity in
$\Z[c_0,\ldots,c_4,\alb a,b,x]$
%--------------------begin array---------------
$$\begin{array}{rcl}
P(x)& = &P(b)-e_2\,P^{[1]}(a) -(2 e_1e_2 + {e_2}^2)\, P^{[2]}(a)
\\
     &   & - (3{e_1}^2 e_2 + 3e_1 {e_2}^2 + {e_2}^3)\,P^{[3]}(b)\\
&   & + (8{e_1}^3e_2+12{e_1}^2{e_2}^2+12e_1{e_2}^3+3{e_2}^4)\,P^{[4]}.
\end{array}$$
%---------------------end array--------------
This gives clearly an evidence that, when $P\in\K[X]$ where $\K$ is an
ordered field,
$$[\,H'(a), \;H'(b), \;a < x < b \,] \; \Longrightarrow \; P(x)>0.$$
\end{example}
%--- end-example-----------------------------------

One can find more information about mixed and generalized Taylor
formulae in \cite{Lom92,War,War2}.
The important thing is that for any
fixed degree, and any combination of signs for $P$ and its derivatives
(which are assumed to be fixed on the interval), there exists a
corresponding GTF.
We state a general result giving the existence of
GTF's.
% -- Proposition{propGTF}-------------
\begin{proposition} \emph{(see \cite{War})}
\label{propGTF}
Let  $P$ be a polynomial of degree $d$ in $\K[X]$ and $a,b,x$ three
variables.
Let  $e_1=x-a$, $e_2=b-x$.
Let $\varepsilon=(\epsilon_1,\ldots ,\epsilon_d)$ be any sequence in
$\left\{-1,+1\right\}$. Let $\epsilon_0=1$. Then there exists an
algebraic identity
%-----------------begin $$----------------
$$ P(x)=P(a_0)+\sum_{k=1}^{d-1}
{\epsilon_k H_{k,\varepsilon}(e_1,e_2)P^{[k]}(a_k)}+
\epsilon_d H_{d,\varepsilon}(e_1,e_2)P^{[d]}
$$
%-----------------end $$------------------
where each polynomial $H_{k,\varepsilon}$ is homogeneous of degree $k$
with nonnegative integer coefficients, $a_k=a$ if $\epsilon_k
\epsilon_{k+1}=1$, and
$a_k=b$ if $\epsilon_k \epsilon_{k+1}=-1$.

Moreover, if $\epsilon_1=1$, then $e_1$ divides all the
$H_{k,\varepsilon}$'s,
and the coefficient of $e_1^k$ in $H_{k,\varepsilon}$ is
nonzero.
In a similar way  if $\epsilon_1=-1$, then $e_2$ divides all the
$H_{k,\varepsilon}$'s, and the coefficient of $e_2^k$ in
$H_{k,\varepsilon}$ is nonzero.
\end{proposition}
%--- end-proposition----------------------------------
%-- Remark{rempropGTF}----------------
\begin{remark}
\label{rempropGTF}
Let  $P$ be a polynomial of degree $d$ in $K[X]$ and
let  $a<b\in\Krc$ be such that
$P^{[k]}(a)P^{[k]}(b)\ge 0$  for $k=0,\ldots,d $.
This gives a system of signs
$(\sigma_0,\sigma_1,\ldots ,\sigma_d)$
($\sigma_i=\pm1$) ($\sigma_k$ is the sign of $P^{[k]}(x)$  on the open interval
$\,]a,b[\,$). Let $\epsilon_i=\sigma_0\sigma_i$,
$\varepsilon=(\epsilon_1,\ldots ,\epsilon_d)$.
Then the  corresponding GTF gives an algebraic certificate for the
fact that $\sign(P(x))=\sigma_0$ when $a<x<b$.
\end{remark}
%--- end-remark-----------------------------------------

We now give four GTF's in degree 3, those beginning by $P(a)+
e_1P^{[1]}\cdots $.
Each formula is given also with $e_1$ in factor in the second part.
$$\begin{array}{rl}
P(x)  = &
P(a) + e_1P^{[1]}(a) + e_1^2P^{[2]}(a)+e_1^3P^{[3]}\\
  = &
P(a) + e_1 \left(P^{[1]}(a) + e_1P^{[2]}(a)+e_1^2P^{[3]} \right) \\
&\\
  = &
P(a) + e_1P^{[1]}(a) + e_1^2P^{[2]}(b)
-\left(2e_1^3+e_1^2e_2\right)P^{[3]}\\
  = &
P(a) + e_1 \left( P^{[1]}(a) + e_1P^{[2]}(b)
-\left(2e_1^2+e_1e_2\right)P^{[3]}\right) \\
&\\
    = &
P(a) + e_1P^{[1]}(b) - \left(e_1^2 + 2e_1e_2\right)P^{[2]}(a)
- \left(2e_1^3 + 6e_1^2e_2+3e_1e_2^2\right)P^{[3]}\\
  = &
P(a) + e_1 \left( P^{[1]}(b) - \left(e_1 + 2e_2\right)P^{[2]}(a)
- \left(2e_1^2 + 6e_1e_2+3e_2^2\right)P^{[3]}\right) \\
&\\
    = &
P(a) + e_1P^{[1]}(b) - \left(e_1^2+2e_1e_2\right) P^{[2]}(b)
+ \left(e_1^3+3e_1^2e_2+3e_1e_2^2\right) P^{[3]}\\
  = &
P(a) + e_1 \left( P^{[1]}(b) - \left(e_1+2e_2\right) P^{[2]}(b)
+ \left(e_1^2+3e_1e_2+3e_2^2\right) P^{[3]}\right).
\end{array}$$
  There are also four other GTF's beginning by $P(b)-e_2.P^{[1]}\cdots
$. They can be obtained from the first ones by  swapping $a$ and $b$,
and replacing $e_1$ and $e_2$ by $-e_2$ and $-e_1$
$$\begin{array}{rl}
P(x)  = &
P(b) - e_2P^{[1]}(b) + e_2^2P^{[2]}(b)-e_2^3P^{[3]}\\
  = &
P(b) - e_2 \left(P^{[1]}(b) - e_2P^{[2]}(b)+e_2^2P^{[3]} \right) \\
&\\
  = &
P(b) - e_2P^{[1]}(b) + e_2^2P^{[2]}(a)
+\left(2e_2^3+e_2^2e_1\right)P^{[3]}\\
  = &
P(b) - e_2 \left( P^{[1]}(b) - e_2P^{[2]}(a)
-\left(2e_2^2+e_2e_1\right) P^{[3]}\right) \\
&\\
    = &
P(b) - e_2P^{[1]}(a) - \left(e_2^2 + 2e_2e_1\right) P^{[2]}(b)
+ \left(2e_2^3 + 6e_2^2e_1+3e_2e_1^2\right) P^{[3]}\\
  = &
P(b) - e_2 \left( P^{[1]}(a) + \left(e_2 + 2e_1\right) P^{[2]}(b)
- \left(2e_2^2 + 6e_2e_1+3e_1^2\right) P^{[3]}\right) \\
&\\
    = &
P(b) - e_2P^{[1]}(a) - \left(e_2^2+2e_2e_1\right) P^{[2]}(a)
+ \left(e_2^3-3e_2^2e_1+3e_2e_1^2\right) P^{[3]}\\
  = &
P(b) - e_2 \left( P^{[1]}(a) + \left(e_2+2e_1\right)P^{[2]}(a)
+ \left(e_2^2+3e_2e_1+3e_1^2\right) P^{[3]}\right).
\end{array}$$

%---- paragraph {using gTF}----------------

\ms\textbf{Using generalized Taylor formulae for computing
the variations of the valuation $v(P(x))$.}
%-----------------------------------

Now let us see in the case of an \ovf how these
formulae can be used
in order to describe the variations of $v(P(x))$
when $x$ is on the real line $\Krc$.

%-- Example{exa3}-------------------
\begin{example}
\label{exa3}
  Let $a,b\in \Krc$ and assume that the signs of the derivatives
  of a polynomial $P$ of degree $4$ are the same in $a$ and $b$,
as in  Example \ref{exa2}.
If $x\in [a,b]$ let  $x=a+t_1(b-a)$, $e=b-a$, $e_1=t_1e$, $e_2=t_2e$
(so $t_2=1-t_1$), $\delta=v(e)$, $\tau_1=v(t_1)$, $\tau_2=v(t_2)$,
$\nu_0=v(P(b))$, $\nu_1=v(P^{[1]}(a))$, $\nu_2=v(P^{[2]}(a))$,
$\nu_3=v(P^{[3]}(b))$, $\nu_4=v(P^{[4]})$. We rewrite the GTF as
%--------------------begin array---------------
$$\begin{array}{rcl}
P(x)& =   & P(b) - e\,{t_2}\,P^{[1]}(a)-e^2(2t_1{t_2}+{t_2}^2)\,P^{[2]}(a)\\
&& \qquad  -e^3( 3{t_1}^2t_2 + 3t_1 {t_2}^2 +{t_2}^3)\,P^{[3]}(b)   \\
&&\qquad
+ e^4(8{t_1}^3t_2+ 12{t_1}^2{t_2}^2+
12t_1{t_2}^3+ 3{t_2}^4)\,P^{[4]}.
\end{array}$$
%---------------------end array--------------
In the above GTF, since all terms of the sum are $\ge 0$, the valuation of
the sum is the minimum of valuations of the terms, so we get:
%-----------------begin item------------------
\begin{itemize}
\item [$(1)$] If
$t_1$ and $t_2$ are units, then $\tau_1=\tau_2=0$ and
$ v(P(x))$ is constant equal to
%-----------------begin $$----------------
$$ v(P(x))= \min
(\nu_0,\,\nu_1+\delta,\,\nu_2+2\delta,\,\nu_3+3\delta,\,
\nu_4+4\delta).
$$
%-----------------end $$------------------
\item [$(2)$] If $t_1$ is infinitely close to $0$, then $\tau_1>0$
(decreasing as $t_1$ increases), $\tau_2=0$, and $ v(P(x))$ is a priori
increasing ``piecewise linearly  w.r.t.\  $\tau_1$", but in our case
constant
%-----------------begin $$----------------
$$ v(P(x))= \min (\nu_0,\, \nu_1+\delta,\, \nu_2+2\delta ,\,
\nu_3+3\delta,\, \nu_4+4\delta).
$$
%-----------------end $$------------------
\item [$(3)$] If $t_1$ is infinitely close to $1$, then $\tau_1=0$,
$\tau_2>0$ (increasing as $t_1$ increases), and $ v(P(x))$ is
increasing ``piecewise linearly w.r.t.\  $\tau_2$"
%-----------------begin $$----------------
$$ v(P(x))= \min (\nu_0,\, \nu_1+\delta+\tau_2,\, \nu_2+2\delta+\tau_2
,
\, \nu_3+3\delta +\tau_2,\, \nu_4+4\delta+\tau_2).
$$
%-----------------end $$------------------
In fact here we see that this formula is true in
the three cases and
that only two slopes (w.r.t.\  the variable $\tau_2$)
can appear since
%-----------------begin $$----------------
$$ v(P(x))= \min (\nu_0,\,\min (\nu_1+\delta,\, \nu_2+2\delta, \,
\nu_3+3\delta,\, \nu_4+4\delta)+ \tau_2
).
$$
%-----------------end $$------------------
\end{itemize}
%-----------------end item------------------
\end{example}
%--- end-example-----------------------------------

%-- Example{exa4}-------------------
\begin{example}
\label{exa4}
In a similar way let us see what is given by the second GTF in degree
$3$
%-----------------begin $$----------------
$$ P(x)= P(a) + e_1\,P^{[1]}(a) + e_1^2\,P^{[2]}(b)
-\left(2e_1^3+e_1^2e_2\right)\,P^{[3]}.$$
%-----------------end $$------------------
We assume $P(a)\geq 0$, $P^{[1]}(a)\geq 0$, $P^{[2]}(b)\geq 0$,
$P^{[3]}<0$,
$x=a+t_1(b-a)$, $e=b-a$, $e_1=t_1e$, $e_2=t_2e$
($t_2=1-t_1$), $\delta=v(e)$, $\tau_1=v(t_1)$, $\tau_2=v(t_2)$,
$\nu_0=v(P(a))$, $\nu_1=v(P^{[1]}(a))$, $\nu_2=v(P^{[2]}(b))$, 
$\nu_3=v(P^{[3]})$, and we get
%-----------------begin $$----------------
$$ P(x)= P(a) + e\,t_1\,P^{[1]}(a) + e^2t_1^2\,P^{[2]}(b)
-e^3\left(t_1^3+t_1^2t_2\right)\,P^{[3]}.$$
%-----------------end $$------------------
%-----------------begin item------------------
\begin{itemize}
\item [$(1)$] If
$t_1$ and $t_2$ are units, then $\tau_1=\tau_2=0$ and $ v(P(x))$ is
constant equal to
%-----------------begin $$----------------
$$ v(P(x))= \min
(\nu_0,\,\nu_1+\delta,\,\nu_2+2\delta,\,\nu_3+3\delta).
$$
%-----------------end $$------------------
\item [$(2)$] If $t_1$ is infinitely close to $1$, then $\tau_1=0$,
$\tau_2>0$ (increasing as $t_1$ increases), and $ v(P(x))$ is
increasing ``piecewise linearly w.r.t.\  $\tau_2$", but in our case
constant
%-----------------begin $$----------------
$$ v(P(x))= \min (\nu_0,\, \nu_1+\delta,\, \nu_2+2\delta,
\, \nu_3+3\delta).
$$
%-----------------end $$------------------
\item [$(3)$] If $t_1$ is infinitely close to $0$, then $\tau_1>0$
(decreasing as $t_1$ increases), $\tau_2=0$, and $ v(P(x))$ is
increasing ``piecewise linearly  w.r.t.\  $\tau_1$",
%-----------------begin $$----------------
$$ v(P(x))= \min  (\nu_0,\,\nu_1+\delta+\tau_1,\,\nu_2+2\delta+2\tau_1
,\,\nu_3+3\delta+2\tau_1).
$$
%-----------------end $$------------------
In fact here we see that this formula is true in the three cases and
that only three slopes (w.r.t.\  the variable $\tau_1$) can appear since
%-----------------begin $$----------------
$$ v(P(x))= \min (\nu_0,\, (\nu_1+\delta)+\tau_1,\,
\min(\nu_2+2\delta,\,\nu_3+3\delta)+2\tau_1).
$$
%-----------------end $$------------------
\end{itemize}
%-----------------end item------------------
\end{example}
%--- end-example-----------------------------------

What we have seen on our two
Examples \ref{exa3} and \ref{exa4} is a
general result, that we immediately get as
a corollary of Proposition \ref{propGTF}.
%-- Theorem{thValThom}----------------
\begin{theorem}
\label{thValThom}
Let  $P$ be a polynomial of degree $d$ in $K[X]$ and $a<b\in\Krc$ such
that $P^{[k]}(a)P^{[k]}(b)\ge 0$  for $k=0,\ldots,d $.
Let $(\sigma_0,\sigma_1,\ldots ,\sigma_d)$ be the signs of
$P,P^{[1]},\ldots,P^{[d]}$ in the interval $\,]a,b[\,$
($\sigma_i=\pm1$). Let $\epsilon_i=\sigma_0\sigma_i$,
$\varepsilon=(\epsilon_1,\ldots ,\epsilon_d)$.
Let us consider the corresponding  GTF as in Proposition \ref{propGTF}, and let
us follow the notation there.
Let  $\nu_k=v(P^{[k]}(a_k))$ for $k=0,\ldots,d$. Recall
that
$a_k=a$ if $\epsilon_k \epsilon_{k+1}=1$, and
$a_k=b$ if $\epsilon_k \epsilon_{k+1}=-1$.
Note also that $\nu_k$ may be infinite if $k<d$.
If $x\in [a,b]$ let  $x=a+t_1(b-a)$, $e=b-a$,  $t_2=1-t_1$,
$\delta=v(e)$, $\tau_1=v(t_1)$, $\tau_2=v(t_2)$.

Then for $x\in [a,b]$ the valuation $ v(P(x))$  is monotonic w.r.t.\  $t_1$ and
more
precisely can be described in the following way.
%-----------------begin item------------------
\begin{itemize}
\item [$(a)$]
%-----------------begin item------------------
\begin{itemize}
\item  If $\epsilon_1=1$ we can extract from the GTF integers
$k_2,\ldots,k_d$ such that $1\leq k_j\leq j$  and
%-----------------begin $$----------------
$$ v(P(x))= \min (\nu_0,\,
\nu_1+\delta+\tau_1,\,\nu_2+2\delta+k_2\tau_1,\ldots,
\nu_d+d\delta+k_d\tau_1).
$$
%-----------------end $$------------------
\item  If  $\epsilon_1=-1$ we can extract from the GTF integers
$k_2,\ldots,k_d$ such that $1\leq k_j\leq j$  and
%-----------------begin $$----------------
$$ v(P(x))= \min (\nu_0,\,
\nu_1+\delta+\tau_2,\,\nu_2+2\delta+k_2\tau_2,\ldots,
\nu_d+d\delta+k_d\tau_2).
$$
%-----------------end $$------------------
\end{itemize}
\item [$(b)$]
So in any case the valuation $ v(P(x))$  is
%-----------------begin item------------------
\begin{itemize}
\item  either constant (if $v(P(a))=v(P(b))$),
\item  or increasing piecewise linearly w.r.t.\
$\tau_1=v({x-a \over b-a})$
(if $v(P(a))>v(P(b))$),
\item  or increasing piecewise linearly w.r.t.\
$\tau_2=v({b-x \over b-a})$,
(if $v(P(a))<v(P(b))$).
\end{itemize}
%-----------------end item------------------
\item [$(c)$] Introducing
$$\tau=\tau_1-\tau_2=v\left({t_1\over1-t_1}\right)$$
we also get: $\tau_1=\max(\tau,0)=\tau^+$,
$\tau_2=\max(-\tau,0)=\tau^-$,
and the value $v(P(x))$  is monotone and piecewise linear
w.r.t.\  $\tau$. More precisely,  we can extract from the GTF integers
$k_2,\ldots,k_d$ such that $1\leq k_j\leq j$  and
%-----------------begin $$----------------
$$ v(P(x))= \min (\nu_0,\,
\nu_1+\delta+\tau',\,\nu_2+2\delta+k_2\tau',
\ldots, \nu_d+d\delta+k_d\tau')
$$
where $\tau'=\max(\epsilon_1\tau,0)$.
%-----------------end $$------------------
\end{itemize}
%-----------------end item------------------
\end{theorem}
%--- end-theorem-----------------------------------------

%--- SUBsection{subsecConstSubLine}---------------------
\subsection{Constructible subsets of the real line}
\label{subsecConstSub}
%-----------------------------------------

We introduce here the notion of {\em \vco sets}  in the real valuative
affine space.
This notion corresponds to sets that are definable in the
language of ordered valued fields.
These sets are analogous to Zariski-constructible sets in
algebraic geometry and to semi-algebraic sets in real
algebraic geometry.

%-- Definition{defGSC}-----------------
\begin{definition}
\label{defGSC} Let $\KVP$ be an \ovfz, and consider a finite
family $(x_j)_{j=1,\ldots,m}$ of elements of $\Krc$.
Let us call a {\em valued sign condition (a vsc fort short) for the
  family}  any condition of the following type
%-----------------begin $$----------------
$$ \Land\nolimits_{j\in J} \; \sign(x_j)=\sigma_j
\quad \land\quad
\Land\nolimits_{\ell\in L} \;\sign
\bigg(
\sum\nolimits_{j\in J,\; \sigma_j\neq 0 } \ell_j v(x_j) 
\bigg) =\sigma'_{\ell}
$$
%-----------------end $$------------------
where $J\subseteq \{1,\ldots,m \}$, $\ell\in L$ ($L$ is a finite subset
of $\Z^{\{j\;:\;j\in J,\,\sigma_j\neq 0 \}}$) and
$\sigma_j,\sigma'_{\ell}\in \left\{-1,0,1\right\} $.

Let $N$ be a positive integer.
We call an {\em $N$-complete system of valued sign
conditions on the family $(x_j)_{j=1,\ldots,m}$} 
a system of vsc's that gives
all the signs $\sign(x_j)$ and all the signs
$\sign\left(\sum_{x_j\neq 0}\ell_jv(x_j) \right)$  for all
$\ell\in\{-N,\ldots,0,\ldots,N\}^{\{ j\; :\; 1\leq j\leq m,\,x_j\neq 0\}}$.
\end{definition}
%--- end-definition------------------------------------

An alternative definition could use
$\sign\left(\sum_{j\in J,\; } \ell_j v(x_j) \right)$  even when
$v(x_j)=\infty $ for some $j$'s. But there should be no natural way
to give a sign to an expression containing $\infty-\infty$.

%-- Definition{defbasicsa}-----------
\begin{definition}
\label{defbasicsa} Let $\KVP$ be an \ovsf of a \rcvf
$\RVP$, and consider a finite family
$(P_j)_{j=1,\ldots,m}$ of polynomials in  $\K[X_1,\ldots,X_n]$.
%-----------------begin item------------------
\begin{itemize}
\item  The subset of $\R^n$ made of the
$\underline{x}=(x_1,\ldots ,x_n)$ such that the
$P_j(\underline{x})$'s verify some given system of vsc's
is called a {\em basic \vco set} defined over $\KVP$.
\item  A {\em (general)  \vco set  defined over $\KVP$} is any
boolean combination $S$ of basic
\vco sets defined over $\KVP$. If $(P_j)_{j=1,\ldots,m}$ is a
family of polynomials such that any basic component of $S$ is defined as in
the first item, we say  that {\em $S$ is described from
$(P_j)_{j=1,\ldots,m}$}.
\item  Let $S\subseteq \R^n$ be a \vco set. A map
$f:S\rightarrow \R^p$ is called a \emph{\vco map} if its graph
is a \vco subset of $\R^{n+p}$.
\end{itemize}
%-----------------end item------------------
\end{definition}
%--- end-definition------------------------------------
Let us recall that the order topology and the valued topology are
identical in a \rcvfz.

%-- Notation{notaVcoLine}--------------
\begin{notation}
\label{notaVcoLine}
Let $\KVP$ be an \ovsf of a \rcvf
$\RVP$. We shall use the following notations for some
convex open \vco subsets  of the real line. They are basic \vco sets  defined
over $\KVPrc$.
%--------------------begin array---------------
$$\begin{array}{rcl}
\I^+(a,\alpha)& =  & \left\{ x\in\R\; :\;x=a+t,\; 0<t,\; v(t)=\alpha
\right\} \\
&    & \; \; \mathrm{with} \;  a\in\Krc,\; \alpha\in\GKd.  \\
\I^-(a,\alpha)& =  & \left\{ x\in\R\; :\;x=a-t,\; 0<t,\; v(t)=\alpha
\right\} \\
&    & \; \; \mathrm{with} \;a\in\Krc,\; \alpha\in\GKd.   \\
\I^+(a,\alpha,\beta)& =  & \left\{ x\in\R\; :\;x=a+t,\; 0<t,\;
\alpha<v(t)<\beta \right\}  \\
&    & \; \; \mathrm{with} \;a\in\Krc,\; \alpha<\beta\;\mathrm{in}
\;\GKd \cup \{\pm\infty \}. \\
\I^-(a,\alpha,\beta)& =  & \left\{ x\in\R\; :\;x=a-t,\; 0<t,\;
\alpha<v(t)<\beta \right\}   \\
&    & \; \; \mathrm{with} \;a\in\Krc,\;\alpha<\beta\;\mathrm{in}
\;\GKd \cup \{\pm\infty \}.  \\
\J^+(a,b,\alpha)& =  & \left\{ x\in\R\; :\;x=a+t(b-a),\; 0<t,\;
v(t)=\alpha \right\}   \\
&    & \; \; \mathrm{with} \;a<b\in\Krc,\; 0<\alpha\in\GKd.  \\
\J^-(a,b,\alpha)& =  & \left\{ x\in\R\; :\;x=b-t(b-a),\; 0<t,\;
v(t)=\alpha \right\}   \\
&    & \; \; \mathrm{with} \;a<b\in\Krc,\; 0<\alpha\in\GKd. \\
\J^+(a,b,\alpha,\beta)& =  & \left\{ x\in\R\; :\;x=a+t(b-a),\; 0<t,\;
\alpha<v(t)<\beta \right\}   \\
&    & \; \; \mathrm{with} \;a<b\in\Krc,\; 0\leq \alpha<\beta\in\GKdi.  \\
\J^-(a,b,\alpha,\beta)& =  & \left\{ x\in\R\;:\;x=b-t(b-a),\; 0<t,\;
\alpha<v(t)<\beta \right\}   \\
& & \; \; \mathrm{with} \;a<b\in\Krc,\; 0\leq \alpha<\beta\in\GKdi.  \\
\J(a,b)& =  & \left\{ x\in\R\; :\;x=a+t(b-a),\; 0<t<1,\;
v(t)=v(1-t)=0 \right\}   \\
&    & \; \; \mathrm{with} \;a<b\in\Krc.
\end{array}$$
%---------------------end array--------------
These subsets will be called {\em \vcis  defined over $\KVP$}.
\end{notation}
%--- end-notation-----------------------------------------

\sni \textbf{Some remarks.}
%-----------------begin item------------------
\begin{itemize}
\item  The subsets of $\R$ given in definition
\ref{defbasicsa} are a priori basic \vco sets  defined over $\KVPrc$.
But they are also general \vco sets defined over $\KVP$:
this is a consequence of Remark \ref{rempropCHV2}~(2).
\item  In $\J^+(a,b,\alpha)$, $\J^-(a,b,\alpha)$,
$\J^+(a,b,\alpha,\beta)$  and
$\J^-(a,b,\alpha,\beta)$ we have $0<t<1$ (in fact $t<$ any positive
rational number) since $t>0$ and $v(t)>0.$
\item  Except when $\beta=\infty$, any \vci is closed.
\item  We have
%--------------------begin array---------------
$$\begin{array}{rcl}
   ]a,\infty [\, & = &  \I^+(a,-\infty,\infty),  \\
   ]a,b[\, & = &  \J^+(a,b,0,\infty)\;\cup\;\J(a,b)\;\cup\;
\J^-(a,b,0,\infty),  \\
   \I^+(a,\alpha,\gamma) & = &  \I^+(a,\alpha,\beta)\;\cup\;
\I^+(a,\beta)\;\cup\;\I^+(a,\beta,\gamma) \; \; \; \mathrm{if} \;
\alpha<\beta<\gamma,
\end{array}$$
%---------------------end array--------------
and similar results with $\I^-$, $\J^+$ and $\J^-$.
\item  When $t>0,\; \alpha<\beta\in\GKdi,\; c>0\in\K$ and 
$v(c)=\alpha+\beta$ we
have the following equivalences
%--------------------begin array---------------
$$\begin{array}{rcll}
\alpha<v(t)<\beta &\Longleftrightarrow &\alpha<\min(v(t),v(c/t))
&\Longleftrightarrow\\
\alpha<v(t+c/t)&\Longleftrightarrow&
\alpha+v(t)<v(t^2+c).
\end{array}$$
%---------------------end array--------------
\item Concerning $\J(a,b)$ we have
%--------------------begin array---------------
$$\begin{array}{rcl}
\J(a,b)& =  & \left\{ x\in\R\; :\;x=a+t(b-a),\; 0<t(1-t),\;
v(t(1-t))= 0 \right\}.
\end{array}$$
%---------------------end array--------------
\item  All  $\J$'s could be considered as particular cases of  $\I$'s, e.g.,
$\J^+(a,b,\alpha,\beta)=\I^+(a,\alpha',\beta')$ with
$\alpha'=\alpha+v(b-a)$ and $\beta'=\beta+v(b-a)$.
\item  We could introduce
$$\begin{array}{rcl}
\J(a,b,\alpha)& =  & \left\{ x\in\R\; :\;x=a+t(b-a),\; 0<t<1,\;
v(t/(1-t))=\alpha \right\}   \\
&    & \; \; \mathrm{with} \;a<b\in\Krc,\; \alpha\in\GKd, \\
\J(a,b,\alpha,\beta)& =  & \left\{ x\in\R\; :\;x=a+t(b-a),\; 0<t<1,\;
\alpha<v(t/(1-t))<\beta \right\}    \\
&    & \; \; \mathrm{with} \;a<b\in\Krc,\;\alpha<\beta\;\mathrm{in}
\;\GKd \cup \{\pm\infty \}.
\end{array}$$
We should have $\J^+(a,b,\alpha)=\J(a,b,\alpha)$,
$\J^-(a,b,\alpha)=\J(a,b,-\alpha)$,
$\J^+(a,b,\alpha,\beta)=\J(a,b,\alpha,\beta)$,
$\J^-(a,b,\alpha,\beta)=\J(a,b,-\beta,-\alpha)$ and
$\J(a,b)=\J(a,b,0)$.
\end{itemize}
%-----------------end item------------------

An easy corollary of Theorem \ref{thValThom} is the following
description of \vco subsets of the real line.
%-- Theorem{thVcoLine}-----------------
\begin{theorem}
\label{thVcoLine}
Let $\KVP$ be an \ovsf of a \rcvf $\RVP$.
Any   \vco set  of $\R$  defined over $\KVP$ is a finite disjoint union
of points in $\Krc$ and of \vcis defined over $\KVP$ as in
Notations \ref{notaVcoLine}.
  \end{theorem}
%--- end-theorem-----------------------------------------
%-----------------begin proof------------------
\begin{proof}{Proof.}
We give a sketch of the proof on an example.
Assume that the \vco set $S$ is defined from vsc's
on 3 polynomials $P_1,P_2,P_3$ of degrees $5$, introduce all real 
roots of these
polynomials and of all their derivatives. Consider two consecutive
roots $a,b$.  We want to understand what $S\, \cap\,]a,b[\, $ is.

First let us see what
$S\, \cap\, \J^+(a,b,0,\infty) $ looks like. We know that each
$\sign(P_j(x))$ is constant on $\,]a,b[\,$. Concerning the valuations
$v(P_j(x))$,  we know from Examples \ref{exa3} and \ref{exa4}
and Theorem \ref{thValThom}
% something that looks like
that they are piecewise linear functions of
$\tau_1=v(x-a)/v(b-a)$, e.g., of the following forms
%--------------------begin array---------------
$$\begin{array}{rcl}
v(P_1(x))& =  & \min(\mu_0,\mu_1+\tau_1,\mu_2+2\tau_1),\\
v(P_2(x))& =  & \min(\eta_0,\eta_1+\tau_1,\eta_3+3\tau_1),\\
v(P_3(x))& =  & \min(\lambda_0,\lambda_1+\tau_1,\lambda_2+2\tau_1,
\lambda_4+4\tau_1).
\end{array}$$
%---------------------end array--------------

Note that $\tau_1$ varies on $\,]0,+\infty[\,$.
These  piecewise linear functions have polygonal graphs
inside $(\GKd\cap \,]0,+\infty[\,)\times \GKd$.
It is possible to compute the vertices
of these three polygonal graphs.
E.g., if $\lambda_4<\lambda_1<\lambda_0$ and
$3\lambda_2>2\lambda_1+\lambda_4$ we have
two vertices on the polygonal graph
of $v(P_3)$ at the points with coordinates
$$\begin{array}{rcl}
\tau_{1,1}=\beta_1=(\lambda_1-\lambda_4)/3,&\;\;&
v(P_3(x))=\lambda_1+\beta_1=\lambda_4+4\beta_1,\\
\tau_{1,2}=\beta_2=\lambda_0-\lambda_1,&\;\;&  v(P_3(x))=\lambda_0=
\lambda_1+\beta_2.
\end{array}$$

All these vertices give a finite number of valuations for $\tau_1$:
$\alpha_1<\cdots<\alpha_n$. Let $\alpha_0=0$, $\alpha_{n+1}=\infty$.
On each $\J^+(a,b,\alpha_i,\alpha_{i+1})$ ($0\leq i\leq n$) and on
each $\J^+(a,b,\alpha_i)$ ($1\leq i\leq n$), we
know that each $v(P_j(x))$ ($1\leq j\leq 3$) is a fixed ``affine function" of
$\tau_1$. So, the same is true for any linear combination
$$\ell_1 v(P_1(x))+\ell_2 v(P_2(x))+\ell_3 v(P_3(x)),$$
and we can compute % (if any)
the valuation $\tau_1$ for which such an expression
changes sign.

So the intersection $S\, \cap\, \J^+(a,b,0,\infty) $ is a finite
disjoint union of $\J^+(a,b,\alpha,\beta)$ and $\J^+(a,b,\alpha)$
subsets.

In a similar way $S\,\cap \, \J(a,b)\; $ is either empty or equal to $\J(a,b)$,
and  $S\, \cap\, \J^-(a,b,0,\infty) $ is a finite disjoint
union of $\J^-(a,b,\alpha,\beta)$ and $\J^-(a,b,\alpha)$ subsets.

Finally the intersection of $S$ with the final (resp. initial) open
interval is computed in a similar way as a finite union of $\I^+$
(resp. $\I^-$) intervals.
\end{proof}
%-----------------end proof------------------

%--- SECTION{Computing with ordered valued fields}-----
\section{Computing in the real closure of an ordered valued field}
%---------------------------------------
\label{sec OVF}

%--- SUBsection{Codes a la Thom} -----------------------
\subsection{Codes \`a la Thom and valuations in the value group}
\label{subsec.Thom}
%-----------------------------------------------

The real closure $\Krc$ of an ordered field $(\K,\P)$ is unique up to
unique $(\K,\P)$--isomorphism.
This fact allows us to give an explicit construction of the real closure
$\Krc$ (this is ``well-known" from Tarski or even from Sturm and
Sylvester, for a fully constructive proof see \cite{LR}).

E.g., it is possible to describe any element $x$ of $\Krc$ by a
so-called {\em  code \`a la Thom} (see \cite{CCS,CR}):
%-- Definition{defcalatom}-------------
\begin{definition}
\label{defcalatom}
A pair $(P,\sigma)$ where $P\in\K[X]$ is a monic polynomial of
degree $d$ and
$\sigma=(\sigma_1,\ldots, \sigma_{d-1})\in \{1 ,-1 \}^{d-1}$
codes the root $x$ of $P$ in $\Krc$ when one has
%-----------------begin $$------------------
$$P(x)=0\qquad \mathrm{and}\qquad \sigma_i\cdot P^{(i)}(x)\geq 0 \quad
\mathrm{for}\;   i=1,\ldots, d-1.
$$
%-----------------end $$------------------
The pair $(P,\sigma)$ is called a
\emph{code \`a la Thom (over $\K$) for $x$.}
\end{definition}
%--- end-definition------------------------------------

There are algorithms that use only the algebraic structure of
$(\K,\P)$ and give the codes \`a la Thom corresponding to the
roots of $P$ in $\Krc$.
It is possible to make explicit algebraic computations and sign's
tests for such elements that are coded  \`a la Thom.
See e.g., \cite{CCS,CR} or Proposition \ref{propCH}.

On the other hand, the Newton polygon algorithm
  allows us to determine the valuation
  $v(x)$ for any $x$ in the algebraic closure
  of $\K$. How can we match these algorithms?

%--- SUBsection{subsec.co.ovf} Basic comput problems ---
\subsection{Three basic computational problems in the real
closure of an \ovf}
\label{subsec.co.ovf}
%---------------------------------------

Consider an \ovf $\KVP$. Since its real closure (with valuation)
is determined up to unique $\KVP$-isomorphism, the following
computational problems makes sense:

%-- Comput{Real valued closure 1}------
\begin{comput} \label{comput RCVF1}~ \\
Let $\KVP$ be an ordered valued field. \\
\Inp A code \`a la Thom $(P,\sigma)$ over $\K$ for an element $x$ of $\Krc$. \\
\Output The valuation $v(x)$  of $x$ in $\GKdi$.
  More precisely, compute some
$a\in \K$ and a positive integer $n$ such that
$n\times v(x)=v(a)$.
\end{comput}
%---end-comput-----------------------------------

%-- Remark{RCVF1}----------------------
\begin{remark} \label{rem RCVF1}
Assume that the leading coefficient of $P\in\V[X]$ is a unit.
The real zeroes of $P$ are in $\Vrc$.
Let  us denote by  $\overline{x}$ the residue
in $\ResKrc$ of the zero $x$
and by   $\overline{P}$ the residue in $\ResK[X]$
of the polynomial $P$.
Then it is clear that $(\overline{P},\sigma)$ is
a code \`a la Thom over $\ResK$ for $\overline{x}$ since
the residual field $\ResKrc$ can be identified with the
real closure $\ResK^\mathrm{rc}$ of $\ResK$.
\end{remark}
%---end-remark-----------------------------------

More generally, we can ask for algorithms
solving general existential problems.
%-- Comput{Real valued closure 2}------
\begin{comput} \label{comput RCVF2} ~\\
Let $\KVP$ be an \ovfz, and consider a finite family of
polynomials,
$(F_j)_{j=1,\ldots,m}$ in $\K[X]$.
Let $(x_h)_{h=1,\ldots,p}$ be the
ordered family of the zeroes of the $(F_j)$'s in $\Krc$.
Recall that the number $p$ and all the signs $\sign(F_j(x))$,
for $x$ equal to some $x_h$ or inside some
corresponding open interval, can be
determined by computations in the ordered field $(\K,\P)$.\\
\Inp The family $(F_j)_{j=1,\ldots,m}$. \\
\Output
All the valuations
  $v(F_j(x_h))$ $(h=1,\ldots,p)$ and $v(x_{h+1}-x_h)$
$(h=1,\alb\ldots,\alb p-1)$ in $\GKdi$.
\end{comput}
%---end-comput-----------------------------------

%-- Comput{Real valued closure 3}------
\begin{comput} \label{comput RCVF3} ~\\
Let $\KVP$ be an \ovfz.\\
\Inp A finite family $(F_j)_{j=1,\ldots,m}$ in $\K[X]$.
A finite family $(\ell_k)_{k=1,\ldots,r}$ of elements of $\Z^m$.\\
\Output
All occurring  systems of valued sign conditions of the following type
for the family  $(F_j(x))_{j=1,\ldots,m}$  when $x\in\Krc$:
%-----------------begin $$----------------
$$ \left(  (\sign(F_j(x)))_{j=1,\ldots,m},\;
\left(\sign\left(\sum_{j\in \{1,\ldots,m \},\;F_j(x)\neq 0 }
\ell_{k,j}\,v(F_j(x)) \right) \right)_{k=1,\ldots,r} \right).
$$
%-----------------end $$------------------
\end{comput}
%---end-comput-----------------------------------

%-- Remark{remRCVF3}-------------------
\begin{remark}
\label{remRCVF3}
Assume that the family is stable under derivation. From Theorem
\ref{thValThom} (see e.g., the proof of Theorem \ref{thVcoLine})
it is clear that Computational Problem \ref{comput RCVF3} can be solved by
using the solution of Computational Problem \ref{comput RCVF2}.
In fact we can describe in a finite
way all occurring lists
$$\left(  \; (\sign(F_j(x)))_{j=1,\ldots,m}
,\;(v(F_j(x)))_{j=1,\ldots,m}\; \right)$$
when $x\in\Krc$: for $x$ on any \vci $I$ used in the proof of Theorem
\ref{thVcoLine} we have $v(F_j(x))=\mu_{I,j}+m_{I,j}v(t)$
where $t$ is either
$(x-x_h)/(x_{h+1}-x_h)$, or
$(x_{h+1}-x)/(x_{h+1}-x_h)$, or
$x_1-x$ or $x-x_p$.
\end{remark}
%--- end-remark-----------------------------------

%--- SUBsection{subsec.Solve1} --------------------
\subsection{Solving the first problem}
\label{subsec.Solve1}
%----------------------------------------

%-- Algo{RCVF1}----------------------------
\medskip \noindent \textbf{Algorithm RCVF1 solving Problem \ref{comput RCVF1}.}
Recall that $(P,\sigma)$ is a code \`a la Thom for a root $x$ of
$P\in\K[X]$. We can assume w.l.o.g.\  that $P(0)\neq 0$,  $x> 0$
(else replace $P$ by $P(-X)$) and that $P$ is monic.
Let $(x_i)_{i=1,\ldots,d}$ be the roots of $P$ in $\Kac$.
Using the \NPAz, we compute the multiset
$[\,v(x_i)\mid i=1,\ldots,n\,]$.
So we can express the set of valuations $v(x_i)$ as
$(v(c_j)/n_j)_{j=1,\ldots,r}$ for some $r$-tuple
$(c_j,n_j)_{j=1,\ldots,r}$ with $c_j>0$ in $\K$, $n_j\in \N$ and
$v(c_j)/n_j< v(c_{j+1})/n_{j+1}$ for $j=1,\ldots,r-1$.

\noi Consider the LCM $n$ of denominators $n_j$ and
``replace each $x_i$ by $z_i= x_i^{n}$": i.e., compute
$Q(X)=\prod_i(X-z_i)$ and compute a code \`a la Thom $(Q,\sigma')$
for $z= x^{n}$. Let $b_j=c_j^{n/n_j}$.
Then  $v(b_j)=(n/n_j)v(c_j) $ for $j=1,\ldots,r$ and
%-----------------begin $$------------------
$$v(b_1) < \cdots < v(b_{r}).$$
%-----------------end $$------------------
So we have also
%-----------------begin $$------------------
$$b_1 > \cdots > b_r > 0.$$
%-----------------end $$------------------
By rational computations in $(\K,\P)$ we can settle one of the three
following inequalities in  $\Krc$
%--------------------begin array---------------
$$\begin{array}{rl}
&z \geq b_1,       \\
b_{j} \geq z \geq b_{j+1}    &  \mathrm{\; with\; some\; }
j \in \{1,\ldots,r-1 \},   \\
&b_{r} \geq z >0.      \\
\end{array}$$
%---------------------end array--------------
In the first case we conclude that $v(z)=v(b_1)$. In the last case
$v(z)=v(b_r)$. In the remaining case we know that
%-----------------begin $$------------------
$$v(b_{j}) \leq v(z) \leq v(b_{j+1})\qquad \mathrm{ so}\quad
v(z)=v(b_{j})\quad \mathrm{ or}\quad v(z)=v(b_{j+1}).
$$
%-----------------end $$------------------
We have to find  the exact valuation.
Consider $c\in \P $ verifying
$$\displaylines{
0< v(c)\leq \mathrm{ min}\left( v\left({b_{j}\over b_{j-1}}\right),
v\left({ b_{j+1}  \over  b_{j} } \right) \right)  \mathrm{ \quad if\; }
j >1 \cr
0< v(c)= v\left({ b_{2}  \over  b_{1} } \right)  \mathrm{ \quad if\; }
j =1
}$$
(if $j>1$, $c$ can be chosen as ${ b_{j}  /  b_{j-1} }$ or
$ b_{j+1}  /  b_{j} $).
Next consider the linear fractional change of variable
$$y\longmapsto \varphi(y) = { y  \over  1+cy^{2} }$$
We have

\noi
--- If $v(y) \geq 0$ then  $v(\varphi (y))=v(y)$.

\noi
--- If $v(y) \leq -v(c)$ then, letting $y'=1/y$  we get
%-----------------begin $$------------------
$$v(y') \geq v(c)>0, \;
\varphi (y)= { y'  \over  {c+{y'}^2} }\; {\quad \mathrm{and} \quad }  \;
v(\varphi (y))=v(y')-v(c) \geq 0.$$
%-----------------end $$------------------
So the monic polynomial
%-----------------begin $$------------------
$$R(Y)=\prod_i{\left(Y-\varphi\left({z_i\over b_j}\right) \right)}
$$
%-----------------end $$------------------
has coefficients in $\V$.
Moreover $v(z/b_{j})\geq 0$,  so $\varphi(z/b_j)$ is a unit iff
$v(b_{j}) = v(z)$ since $v(\varphi(z/b_{j}))=v(z/b_{j})$.

We can compute a code \`a la Thom $(R,\sigma'')$
for $\varphi (z/b_j)$. This gives a code \`a la Thom
$(\overline{R},\sigma'')$ for $\overline{\varphi (z/b_j)}$
(i.e., $\varphi (z/b_j)$ considered as an element of
$\ResK^\mathrm{rc}$).
Finally we test whether this code is verified by
$\overline{0}$ (which is a root of $\overline{R}$).
In case of negative answer then
$v(z) =v(b_{j})$. Otherwise $v(z) =v(b_{j+1})$.
\eop
%----end-algo---------------------------------------
%-- Remark{AlgoRCVF1}----------------------
\begin{remarks} \label{AlgoRCVF1}~

\noindent 1) In a more explicit view, we should ask for computing two
nonnegative elements $a$ and $b$ of~$\K $ and an integer $n$ such that
$a \leq \vert x \vert^{n} \leq b$ and $v(a)=v(b)$.

\noindent 2) Clearly algorithm \textbf{RCVF1} allows us to run sure computations
inside $\KVPrc$ when we know how to compute inside
$\KVP$.
\end{remarks}
%---end-remark-----------------------------------

%--- SUBsection{subsec.Solve2} -------------------------
\subsection{Solving the second problem}
\label{subsec.Solve2}
%----------------------------------------

First we recall the Cohen-H\"ormander algorithm for ordered fields
(see e.g., \cite{BCR} chapter~1).

%-- Definition{defCTS}------------
\begin{definition}
\label{defCTS}
Let $(\K,\P)$ be an ordered field and $(F_j)$   a finite family of univariate
polynomials in $\K[X]$.
A \emph{complete tableau of signs for the family $(F_j)$} is the
following \emph{discrete data}~$T$:
%-----------------begin item------------------
\begin{itemize}
\item  The ordered list $(x_k)_{k=1,\ldots,r}$ of all the roots of all the
$F_j$'s in $\Krc$.
\item  The signs ($\in\{-1,0,+1\}$) of all the $F_j$'s at all the
$x_k$'s.
\item  The signs of all the $F_j$'s in each interval
$\,]-\infty,x_1[\,$,
$\,]x_k,x_{k+1}[\,$ $(1\leq k\leq r-1)$ and $\,]x_r,+\infty[\,$.
\end{itemize}
%-----------------end item------------------
We call an $x_k$ \emph{a point of the tableau} $T$. Similarly an
interval
$\,]-\infty,x_1[\,$ or $\,]x_k,x_{k+1}[\,$ or
$\,]x_r,+\infty[\,$ is
  called \emph{an interval of the tableau} $T$.
\end{definition}
%--- end-definition------------------------------------

In this tableau $x_k$ is merely a name for the corresponding root,  it may be
coded by the number $k$ or in another way.

%-- Proposition{propCH}-----------
\begin{proposition}
\label{propCH}
\emph{(Cohen-H\"ormander's algorithm for computing the complete tableau
of signs for a finite family of univariate polynomials)}
Let $(\K,\P)$ be an ordered subfield of a \rcf $(\R,\P_\R)$.
Let $L=(F_1,\ldots, F_k)$ be a list of polynomials in $\K[Y]$.
Let $L'$ be the family of  polynomials generated by the
elements of $L$ and by the operations $P\mapsto P'$  and
$(P,Q) \mapsto \mathrm{Rem}(P,Q)$ for $\deg(P) \ge \deg(Q) \ge 1$.
Then $L'$ is finite and one can compute
the complete tableau of signs for $L'$
in terms of  the following data:
%-----------------begin item------------------
\begin{itemize}
\item the degree of each polynomial in the family $L'$,
\item the diagrams of operations $P \mapsto P'$
and $(P,Q) \mapsto \mathrm{Rem}(P,Q)$,
\item the signs of constants $\in L'$.
\end{itemize}
%-----------------end item------------------
\end{proposition}
\begin{proof}{Proof.}
Let us remark that in this algorithm the zero polynomial
can appear in $L'$ as a remainder
$\mathrm{Rem}(P,Q)$ where  $\deg(P) \ge \deg(Q) \ge 1$.
The degree of the zero polynomial is $-1$.

The list $L'$ is finite: one makes systematically the operation
``derivation of every previously obtained polynomial" and
``remainders of all previously obtained couple of polynomials",
and one gets a finite family at the end
since degrees are decreasing.

\smallskip
Let us number the polynomials in $L'$ with an order compatible with
the order on the degrees. Let $L'_m$ be the subfamily of
$L'$ made of polynomials numbered from 1 to $m$. This family is
obviously
stable under the operations ``derivation" and ``remainder by a
division" which decrease strictly the degrees. Denote lastly by $T_m$
the corresponding complete tableau of signs.

We are going to prove, by induction on $m$,
that the complete tableau of signs of
the polynomials in the family $L'_m$ can be obtained by using
only the authorized informations.
As long as polynomials are of degree $0$, this is clear.
Suppose it is true up to $m$.
Let $P$ be the polynomial of number $m+1$ in $L'$.
On each interval of $T_m$, the polynomial $P$ is strictly monotonic.
Every point $a$ of $T_m$ is either $+\infty$, or $-\infty$, or a root
of a certain polynomial $Q$ with number $\le m$, and in this case,
if $R=\mathrm{Rem}(P,Q)$, we have $P(a)=R(a)$.
The sign of $P(a)$ is hence known in every
case from the authorized informations.
This allows us to know on which open intervals of $T_m$
the polynomial $P$ has a root in $\R$.
Let $x$ be such a root of $P$ on one of these open
intervals $I =\,]a,b[\,$. If $Q$ is a polynomial of number
$\le m$ in $P$, its sign on the interval
$I$ is known. This means we know its sign
at the point $x$, and on intervals $\,]a,x[\,$ and $\,]x,b[$.
With respect to $P$,  its signs on $\,]a,x[\,$ and on $\,]x,b[\,$
are also known since $P$ is strictly monotonic on the interval.
The complete tableau of signs for $L'_{m+1}$ is thus known
from the authorized
informations and the complete tableau of signs for $L'_m$.
\end{proof}

In this algorithm we remark that each zero of
the tableau is obtained with a Thom's encoding.

An extension of previous algorithm will solve
Problem \ref{comput RCVF2}. First we give a valued version
for the complete tableau of signs.
%-- Definition{defCTVSC}----------
\begin{definition}
\label{defCTVSC}
Let $\KVP$ be an \ovf and $(F_j)_{j\in J}$ a
finite family of univariate polynomials in $\K[X]$.
A {\em complete tableau of vsc's for the family $(F_j)$} is the
following data $T$:
%-----------------begin item------------------
\begin{itemize}
\item  The ordered list $(x_k)_{k=1,\ldots,r}$ of all the roots of all
the $F_j$'s in $\Krc$.
\item  The complete tableau of signs for the family
$(F_j)_{j\in J}.$
\item  All the valuations $v(x_{k+1}-x_k)$ $(k=1,\ldots,r-1).$
\item  All the valuations $v(F_j(x_k))$ $(j\in J,\; k=1,\ldots,r).$
\end{itemize}
%-----------------end item------------------
\end{definition}
%--- end-definition------------------------------------

%-- Algo{AlgoRCVF2} --------------
\noindent\textbf{Algorithm RCVF2 solving Problem \ref{comput RCVF2}.}
A first possibility is to use algorithm \textbf{RCVF1}.
We think that it is interesting to indicate another possibility which
goes in the same spirit as the Cohen-H\"ormander algorithm for ordered
fields. This gives us also simple proofs for theorems in sections
\ref{sec qela} and \ref{secConstSubSpace}.
Call $(P_j)$ the list $L'$ in Proposition  \ref{propCH}. Call
$(x_{m,k})_{k=1,\ldots,r_m}$ the ordered list of all roots of
$L'_m=(P_j)_{j=1,\ldots,m}$.
We replace in the proof of Proposition \ref{propCH}
the complete tableau of signs $T_m$ of $L'$ by
$S_m=T_m\cup V_m$ where $V_m$ collects the valuations
$v(P_j(x_{m,k}))$  ($j\in\{1,\ldots,m\},\; k\in\{1,\ldots,r_m\}$) and
$v(x_{m,k+1}-x_{m,k})$ ($k\in\{1,\ldots,r_m-1\}$.)

Suppose we have done the job up to $m$.
Let $P=P_{m+1}$ be the polynomial of index
$m+1$ in $L'$. The tableau $T_{m+1}$ is computed as in
Proposition  \ref{propCH}.  It remains to compute
missing informations in $V_{m+1}$.

At every root  $a=x_{m,k}$ of a polynomial $Q=P_\ell$
with index $\ell\le m$, if $R=\mathrm{Rem}(P,Q)$,
we have $P(a)=R(a)$ and  $R$ is in $L'_m$, so
the valuation $v(P(a))$ is known from $V_m$.

Let $x=a+t_1(b-a)$ be a root of $P$ on an  open
interval $I =\,]x_{m,k},x_{m,k+1}[\,=\,]a,b[\,$ of $T_m$.
In order to compute all the $v(P_j(x))_{j=1,\ldots,m}$
it is sufficient to compute
$v(t_1)=\tau_1$ and $v(t_2)=\tau_2$ ($t_2=1-t_1$):
Theorem \ref{thValThom}  says us how to get
the valuations $v(P_j(x))_{j=1,\ldots,m}$
from $V_m$, $\tau_1$ and $\tau_2$.

In order to compute  $\tau_1=v(t_1)$ we use a GTF
that expresses $P(x)=P(a+t_1(b-a))=0$ as
$$P(a)+t_1\cdot\left(\sum_{j=1}^d {\epsilon_j\cdot e^j \cdot
G_{j,\varepsilon}(t_1,t_2) \cdot P^{[j]}(a_j)} \right) \qquad
(a_j=a \; \mathrm{or} \;  b)$$
where $e=b-a$,
$t_1\cdot G_{j,\varepsilon}(t_1,t_2)=H_{j,\varepsilon}(t_1,t_2)$
and
$$\sign(\epsilon_jP^{[j]}(a_j))=\sign(-P(a))\;\;(1\leq j\leq d).$$
Moreover, the valuations $v(P(a)=\nu$, $v(P^{[j]}(a_j))=\nu_j$
and $\delta=v(b-a)$ are known.
 From the properties of $H_{j,\varepsilon}$, we know that
$G_{j,\varepsilon}(t_1,t_2)$ is a unit if $\tau_1=0$,
so its valuation in $\GKdi$ depends only on $\tau_1$.
So we get
$$v(P(a))=\nu=\min(\nu_1+\delta+\tau_1,\,\nu_2+2\delta+k_2\tau_1,
\ldots,\nu_d+d\delta+k_d\tau_1)
$$
($\tau_1\ge 0$, and some $\nu_k$'s may be infinite).
The right hand side is an increasing piecewise linear function of $\tau_1$
so we have a unique and explicit solution $\tau_1$.
With $\mu_i=\nu_i+i\delta$ we precisely get
$$ \tau_1=\max\left(\nu-\mu_1,{\nu-\mu_2 \over k_2}, \ldots,
{\nu-\mu_d \over k_d}\right).
$$
%-----------------end $$------------------

Finally $\tau_2$ is computed analogously and we can fill up $V_{m+1}$.

Remark also that if $x$ is on the last interval
$\,]x_{m,r_m},+\infty[\,=\,]a,+\infty[\,$
of $T_m$, we can compute $v(x-a)$ in a similar way
by using the usual Taylor formula.
\eop
%--- end-algo --------------------------------------
%-- Definition{defQsl}---------------------
\begin{definition}
\label{defQsl}
In an additive divisible ordered group $G$ we consider terms built
from variables $\alpha_j$ by $\Q$-linear combinations and by using the
operations $\min$ and $\max$. We call such a term a
{\em \qsl term}. The function defined by such a term is called a
{\em \qsl function of the $\alpha_j$'s}.
\end{definition}
%--- end-definition------------------

We get the following theorem,
similar to  Proposition \ref{propCH}.
%-- Theorem{propCHV}-----------------------
\begin{theorem}
\label{propCHV}
\emph{(An algorithm \`a la Cohen-H\"ormander for computing the complete
tableau of vsc's  for a finite family  of univariate polynomials)  }
Let $\KVP$ be an \ovsf of a \rcvf $\RVP$.
Let $L=(F_1, \ldots, F_k)$ be a list of polynomials in $\K[Y]$.
Let $L'$ be the (finite) family of polynomials generated by the
elements of $L$
and by the operations $P \mapsto P'$ and
$(P,Q) \mapsto \mathrm{Rem}(P,Q)$ for $\deg(P) \ge \deg(Q) \ge 1$.
Call $(c_j)$ the list of  constants $\in L'$.

Then one can compute the complete tableau of vsc's  for $L'$
in terms of the following data:
%-----------------begin item------------------
\begin{itemize}
\item the degree of each polynomial in the family,
\item the diagrams of operations $P \mapsto P'$
and $(P,Q) \mapsto \mathrm{Rem}(P,Q)$ in $L'$,
\item the signs $\sign(c_j)$,
\item the valuations $v(c_j)$.
\end{itemize}
%-----------------end item------------------
Moreover, all the valuations $v(x_{k+1}-x_k)$
and all the valuations $v(P_j(x_k))$ are given as fixed \qsl functions
of the $v(c_j)$'s: each such \qsl function
is a fixed \qsl term (in the ``variables" $v(c_j)$'s) that depends
only on the complete tableau of signs of $L'$.
\end{theorem}
%--- end-theorem----------------------------------

%-----------------begin proof------------------
\begin{proof}{Proof.}
This theorem is an extension of Proposition \ref{propCH}.
The proof is similar.
In fact we get all results by a close inspection of Algorithm
\textbf{RCVF2}.
\end{proof}
%-----------------end proof------------------

%--- SUBsection{subsec.Solve3} -------------------------
\subsection{Solving the third problem}
\label{subsec.Solve3}
%----------------------------------------

%-- Algo{AlgoRCVF3} --------------
\medskip \noindent \textbf{Algorithm RCVF3 solving Problem \ref{comput RCVF3}.}
We  run Algorithm \textbf{RCVF2} and we apply Theorem \ref{thValThom}:
  see Remark \ref{remRCVF3}.
\eop

%-- Definition{defMCTVSC} --------
\begin{definition}
\label{defMCTVSC}
Let  $(F_j)_{j\in J}$ be  a finite family of univariate polynomials in $\K[X]$
(where $\KVP$ is an \ovfz).
We assume the family to be stable under derivation.
Let $M$ be a positive integer.

An {\em $M$-complete tableau of vsc's for the family $(F_j)$} is the
following {\em discrete data} $T$:
%-----------------begin item------------------
\begin{itemize}
\item  The ordered list $(x_k)_{k=1,\ldots,r}$ of all the roots of all
the $F_j$'s in $\Krc$.
\item  For each $k=1,\ldots,r$, the $M$-complete system of vsc's 
(see Definition \ref{defGSC}) for
the family $(F_j(x_k))_{j\in J}$. 
\item  For each $k=1,\ldots,r-1$
%-----------------begin item------------------
\begin{itemize}
\item  The $M$-complete system of vsc's for the family
$(F_j(x))_{j\in J}$  for $x\in \J(x_k,x_{k+1}).$
\item  A partition of $\,\J^+(x_k,x_{k+1},0,\infty)\,$ as a
finite union of $\,2n_k+1\,$ \vcis
%-----------------begin $$----------------
$$\bigcup_{i=0,n_k}\J^+(x_k,x_{k+1},\alpha_{k,i},\alpha_{k,i+1})
\; \; \cup\;\;
\bigcup_{i=1,n_k}\J^+(x_k,x_{k+1},\alpha_{k,i}),
$$
%-----------------end $$------------------
(where $\alpha_{k,0}=0$  and $\alpha_{k,n_k+1}=\infty$) and for each
  \vci $A$ of this partition,
the $M$-complete system of vsc's for the family
$(F_j(x))_{j\in J}$ which is the same one for any $x\in A$.
\item  A similar data concerning $\J^-(x_k,x_{k+1},0,\infty).$
\end{itemize}
%-----------------end item------------------
\item  Similar data concerning $\I^-(x_1,-\infty,\infty)$ and
  $\I^+(x_r,-\infty,\infty)$.
\end{itemize}
%-----------------end item------------------
\end{definition}
%--- end-definition------------------------------------

In this tableau the $\alpha_{k,i}$'s
($0<\alpha_{k,1}<\cdots<\alpha_{k,n_k}<\infty $)
are purely formal and $n_k$ is the
only relevant information concerning
$\alpha_{k,1},\ldots,\alpha_{k,n_k}$.

We now state a result that precises the output of
Algorithm \textbf{RCVF3}.
%-- Theorem{propMCHV} ------------
\begin{theorem}
\label{propMCHV}
\emph{(An algorithm \`a la Cohen-H\"ormander for computing an
$M$-complete tableau of vsc's  for a finite family  of univariate
polynomials)}\\
Let $\KVP$ be an \ovsf of a \rcvf $\RVP$.
Let $M$ be a positive integer.
Let $L = (F_1, \ldots, F_k)$ be a list of polynomials in
$\K[Y]$.
Let $L'$ be the family of polynomials generated by the
elements of $L$
and by the operations $P \mapsto P' $ and
$(P,Q) \mapsto \mathrm{Rem}(P,Q) $ for $\deg(P) \ge \deg(Q) \ge 1$.
Call $(c_j)$ the list of constants $\in L'$.

Then one can compute the
$M$-complete tableau of vsc's  for $L'$  in terms of the following
data:
%-----------------begin item------------------
\begin{itemize}
\item the degree of each polynomial in the family,
\item the diagrams of operations $P \mapsto P'$
and $(P,Q) \mapsto \mathrm{Rem}(P,Q)$ in $L'$,
\item the $N$-complete system of vsc's  for the family $(c_j)$,
\end{itemize}
%-----------------end item------------------
where $N$ is an integer depending only on $M$ and on the list of
degrees in $L$.
\end{theorem}
%--- end-theorem----------------------------------

%--- SECTION{Quantifier elimination algorithms}-
\section{Quantifier elimination algorithms}
\label{sec qela}
%---------------------------------------

%--- SUBsection subsec.paracomp ------------------------
\subsection{Parametrized computations}
\label{subsec.paracomp}
%----------------------------------------

Algorithms \textbf{RCVF2} and \textbf{RCVF3} are uniform:
they can be run when
coefficients in the initial data are polynomials in other
variables which are called \emph{parameters}
(instead of being in the base field).

A case by case discussion appears,
and the straight-line algorithm is
replaced by a branching one.

We describe this situation as a \emph{parametrized algorithm} dealing
with parametrized univariate polynomials.

%-- Theorem{propCHV2}-------------
\begin{theorem}
\label{propCHV2} \emph{(parametrized version of Theorem
\ref{propCHV})}
Let $\KVP$ be an \ovsf of a \rcvf $\RVP$.
Let $L = (F_1, \ldots, F_k)$ be a list of parametrized univariate
polynomials of degrees  $d_1, \ldots, d_k$ in some variable $X$.
Let us run the algorithm  \textbf{RCVF2}
and let us open two branches in the
computation any time we have to know if a given element is zero
or nonzero when computing a remainder.
Moreover, replace remainders by pseudoremainders in order to
avoid denominators.

Consider the family $(c_j)$ of all ``constants" in all $L'$'s that
appear at the leaves of the tree (these constants are
$\K$--polynomials in the parameters).

Finally consider that the computed valuations $v(x_{k+1}-x_k)$
and  $v(P_j(x_k))$  at any leave of the tree are given as
  \qsl functions of the ``variables" $v(c_j)$'s.

Then this global parametrized algorithm is finite and therefore
gives a finite number of possibilities for its output:
the complete tableau of vsc's  for $L$.

More precisely when the signs of the ``constants" $c_j$'s are known,
the complete tableau of signs is known and all the valuations
$v(x_{k+1}-x_k)$ and  $v(P_j(x_k))$ are given as explicit
  \qsl functions in the ``variables" $v(c_j)$'s.
\end{theorem}
%--- end-theorem----------------------------------
%-----------------begin proof------------------
\begin{proof}{Proof.}
The proof of Proposition \ref{propCH} (Cohen-H\"ormander algorithm) works as
well
in the parametrized case. In each branch so created, the proof of Theorem
\ref{propCHV} works as well.
\end{proof}
%-----------------end proof------------------

%-- Remark{rempropCHV2}-----------
\begin{remarks}
\label{rempropCHV2}~

\noindent 1) An important case is obtained when all coefficients of the $F_i$'s are
independent parameters and
$\KVP=(\Q,\Q,\Q^{\geq 0})$.
This ``generic case"
gives the complete description of all situations occurring with a fixed number
of polynomials of known degrees.

\smallskip \noindent 2) Another interesting particular case is the following one, with
only one parameter subject to certain constraints.
We start with a list of polynomials $L=(F_1, \ldots, F_k)$  in 
$\K[Y]^n$, we get
an extended list $L'$ and the complete tableau of signs.
Let $a$ and $b$ be two
consecutive roots in this tableau.
Now we want to make computations with an element $x$ of the interval
$\,]a,b[\,$.
Consider $x$ as a parameter verifying some sign constraints, namely the Thom's
sign conditions that define $\,]a,b[\,$.
We add the
polynomial $Y-x$ to  $L$ and we run the parametrized version of \textbf{RCVF2}.
Only
one root is added: $x$. The new polynomials appearing are  only ``constants" of
the form $Q(x)$ (where $Q$ is in $L'$).
The process go on only trough one branch.
We get the following result:
the valuations $v(x-a)$ and $v(b-x)$ are given as \qsl
functions of some $v(Q(x))$'s.
 From this we also get a similar result concerning
$v(x-x_j)$ where $x_j$ is any root in the tableau.
Naturally, there is also a parametrized version for this result.
\end{remarks}
%--- end-remark-----------------------------------------

Similarly we have a parametrized version of Theorem
\ref{propMCHV}.
%-- Theorem{propMCHV2} -----------
\begin{theorem}
\label{propMCHV2} \emph{(parametrized version of Theorem
\ref{propMCHV})} \\
Let $L = (F_1, \ldots, F_k)$ be a list of parametrized univariate
polynomials of degrees  $d_1, \ldots, d_k$ in some variable $X$.
Let $M$ be a positive integer.
Let us run the algorithm  \textbf{RCVF2}
and let us open two branches in the
computation any time we have to know if a given element is zero or
nonzero when computing a remainder.
Moreover, replace remainders by pseudoremainders in order to
avoid denominators.
Let us call $(c_j)$ the family of all
``constants" in all $L'$'s that
appear at the leaves of the tree (these constants are
$\K$--polynomials in the parameters).

Finally when applying Theorem \ref{thValThom} in order to get the
output of \textbf{RCVF3} from the one of \textbf{RCVF2},
we open three
branches any time we have to know the sign of some $\Z$-linear
combination of $v(c_j)$'s.

Then this global parametrized algorithm is finite and
therefore gives a finite number of possibilities for its output:
the $M$-complete tableau of vsc's  for $L$.

Moreover, these outputs depend on
the following data:
%-----------------begin item------------------
\begin{itemize}
\item the signs of the ``constants" $c_j$'s,
\item the sign test inside a finite subset of the subgroup
generated by the $v(c_j)$'s; which are exactly divisibility tests
between monomials in the $c_j$'s).
\end{itemize}
%-----------------end item------------------
\end{theorem}
%--- end-theorem ----------------------------------

%-- Remark{rempropMCHV2} ---------
\begin{remark}
\label{rempropMCHV2}
Since the computation in the previous theorem is purely
formal, certain systems of conditions corresponding to the
data given by the two last items may be impossible.
If we want to know what are these impossible systems, we have to
use the quantifier elimination algorithm given
in Theorem \ref{corQEL}.
Nevertheless, one can verify that there is no circular argument.
\end{remark}
%--- end-remark-----------------------------------------

%--- SUBsection{subsec.qela} ---------------
\subsection{Quantifier elimination}
\label{subsec.qela}
%----------------------------------------

We now give some corollaries of previous computations for
quantifier elimination.
We recall that these results are well known, see e.g., \cite{CD}.

We consider the first order theory of
\rcvfs based on the language of ordered fields
$(0,1,+,-,\times,=,\leq)$ to which we add the
predicate $x\preceq y$.
So, all constants and variables represent elements in $\K$ (this
corresponds to our previously explained computability assumptions).

%-- Theorem{thCHV} --------------------

Here is a corollary of Theorem \ref{propMCHV2}.
\begin{theorem}
\label{thCHV}
Let $\Phi(\underline{a},\underline{x}) $
be a quantifier free formula in the first order theory of
\rcvfsz. We view the $a_i$'s as parameters and
the $x_j$'s as variables.
Then one can give a quantifier free formula $\Psi(\underline{a})$
such that the two formulae
$\;\exists \underline{x}\;\Phi(\underline{a},\underline{x})\;$ and
$\;\Psi(\underline{a})\; $
are equivalent in the formal theory.
(The terms appearing in the formulae $\Phi$ and $\Psi$ are
$\Z$--polynomials in the parameters, and, in the case of
$\Phi$, also in the variables.)
%-----------------end item------------------
\end{theorem}
%--- end-theorem-----------------------------------
%-----------------begin proof------------------
\begin{proof}{Proof.}
Use recursively Theorem \ref{propMCHV2} and eliminate the $x_j$'s
one after the other.
\end{proof}
%-----------------end proof------------------

We also get the following corollary.
%-- Theorem{corQEL} -------------------
\begin{theorem}
\label{corQEL} Let $\KVP$ be an \ovsf of a \rcvf
$\RVP$. Assume that the sign test and the divisibility
test are explicit inside $\KVP$. Then there is a uniform quantifier
elimination algorithm for the first
order theory of \rcvfs extending $\KVP$.
\end{theorem}
%-----------------end theorem------------------
%--- end-theorem-----------------------------------

%--- SUBsection{subsec.abstractQela} --------------
\subsection{An abstract form of quantifier elimination}
\label{subsec.abstractQela}
%----------------------------------------
%-- Definition {defSperv} -------------

An abstract form of Theorem \ref{propMCHV2} is the following
theorem, that was given the first time by M.J. De la Puente in \cite{Pue}.

First, we need some definitions of the abstract objects.
\begin{definition}
\label{defSperv}
Let us denote the {\em real-valuative spectrum} of a
commutative ring $A$ by $\Sperv A$: an element of $\Sperv A$
is given by a ring
homomorphism $\varphi$ from $A$ to a \rcvf $K$, and two such
homomorphisms $\varphi$, $\varphi'$ define the same element of
$\Sperv A$ iff there exists an isomorphism of \ovfs
$\psi:R\rightarrow R'$ such that $\psi\circ\varphi=\varphi'$, where
$R$ and $R'$ are the \rcvfs generated by $\varphi(A)$ and
$\varphi'(A)$. Alternatively, an element of $\Sperv A$ is given by a prime
ideal $Q$ of $A$ and a structure of \ovf upon the fraction field of $A/Q$.
A {\em constructible subset} of $\Sperv A$ is by
definition a boolean combination of {\em elementary constructible subsets}
$U_x:=\left\{ \varphi\in \Sperv A\; :\; \varphi(x)>0 \right\}$ and
$V_{x,y}:=\left\{ \varphi\in \Sperv A\; :\; \varphi(x)\preceq
\varphi(y) \right\}$, where $x,y\in A$.
\end{definition}
%--- end-definition------------------------------------
%-- Theorem{thSperv}-------------------
\begin{theorem}
\label{thSperv}
The canonical mapping from $\Sperv A[X]$ to
$\Sperv A$
transforms any \vco subset into a \vco subset.
\end{theorem}
%--- end-theorem-----------------------------------------

%-----------------begin proof------------------
\begin{proof}{Proof.}
A \vco subset in $\Sperv (B)$ is a finite union of basic
\vco subsets,
that are defined as
%--------------------begin array---------------
$$\begin{array}{l}
\left\{\varphi\in\Sperv (B)\;:\; \right.  \\
\;\left.\bigwedge_i \varphi(a_i)=0\land \bigwedge_j \varphi(b_j)>0\land
\bigwedge_k v(\varphi(c_k))=v(\varphi(d_k))\land
\bigwedge_\ell v(\varphi(e_\ell))>v(\varphi(f_\ell))
\right\}
\end{array}$$
%---------------------end array--------------
where conjunctions are finite and all elements are in $B$.
Searching the canonical image of a basic constructible subset $S$
of $\Sperv A[X]$
(defined by elements $a_i,b_j, c_k,d_k,e_\ell,f_\ell$ in $A[X]$)
inside $\Sperv A$, is the same thing that analyzing the conditions
on the coefficients of the
polynomials $a_i,b_j, c_k,d_k,e_\ell,f_\ell$ allowing the existence of an
$x$ where the defining conditions of $S$ are verified. So Theorem
\ref{propMCHV2}
gives the answer.
\end{proof}
%-----------------end proof------------------

Another consequence of Theorem \ref{propMCHV2}
is a \emph{relativized version} of Theorem \ref{thSperv}.
This generalization is obtained by giving some \emph{constraints} on the
ring homomorphism $\varphi$ from $A$ to a \rcvf $K$.
We give e.g., a subring $B$ of
$A$, an ideal $M$ of $B$, a multiplicative monoid $S$ in $A$
and a semi ring $P$ in
$A$ ($P+P\subseteq P,\;P\times P\subseteq P$).
We want to allow only homomorphisms
$\phi$ (from $A$ or $A[X]$ to a \rcvfz)  verifying that $\phi(B)$ is in the
valuation ring, $\phi(M)$ is in the maximal ideal, elements of $\phi(S)$ are
nonzero and  elements of $\phi(P)$ are nonnegative.
If we write $C$ the constraints
$(B,M,S,P)$ and if we write $\Sperv(A,C)$ the part of $\Sperv A$ satisfying
the constraints, we get: the canonical mapping from $\Sperv (A[X],C)$ to
$\Sperv(A,C)$ transforms any \vco subset in a \vco subset.

In \cite{Pue} the relativized version is settled with one constraint $B$.

%--- SECTION{secConstSubSpace}-------------------------
\section{Constructible subsets in the real valuative affine space}
\label{secConstSubSpace}
%-----------------------------------------
%--- SUBsection{subsecTSC}-------------------------
\subsection{Tarski-Seidenberg-Chevalley}
\label{subsecTSC}
%-----------------------------------------

%-- Theorem{thConsProj}--------------

We now give a geometric form for Theorems \ref{propMCHV2} and \ref{corQEL}.
\begin{theorem}
\label{thConsProj}
Let $\KVP$ be an \ovsf of a \rcvf $\RVP$.
Let $\pi$ the canonical projection from $\R^{n+r}$ onto  $\R^{n}$.
Let $S\subseteq \R^{n+r}$ be any \vco set  defined over $\KVP$.
Assume that the sign test and the divisibility test are explicit
inside the ring generated by the coefficients of the polynomials
that appear in the definition of $S$.
Then a description of the projection $\pi(S)\subseteq\R^{n}$
can be computed in a uniform way by an algorithm that uses only
rational computations, sign tests and divisibility tests.

In particular, the complexity of a description of $\pi(S)$
is explicitly bounded in terms of the complexity of a description of
$S$.
\end{theorem}
%--- end-theorem-----------------------------------------

Here \emph{rational computations} mean computations in the ring generated by
the coefficients of the polynomials occurring
in the description of $S$.
A \emph{description of $S$} is a quantifier free
formula in disjunctive
normal form describing $S$.
The \emph{complexity} of such a description of
$S$ can be defined as a $5$-tuple $(n,d,k,\ell,m)$ where $n$ is the
number of variables, $d$ is the maximum of the degrees, $k$ is the
number of polynomials, $\ell$ is the number of $\lor$
and $m$ is the
bound for the numbers of $\land$ inside a disjunct.

%-- Corollary {corQEL2} ---------------
\begin{corollary}
\label{corQEL2} Let $\KVP$ be an \ovsf of a \rcvf
$\RVP$. Let $S\subseteq \R^n$ be a \vco set and let
$f:S\rightarrow \R^p$ be a  \vco map.
%-----------------begin item------------------
\begin{itemize}
\item  The interior and the adherence of $S$ inside $\R^n$
for the order topology are \vco sets.
\item  $f(S)\subseteq \R^p$ is a \vco set.
\item  Let $T$  be a \vco set containing $f(S)$ and let
$g:T\rightarrow \R^q$ be a  \vco map. Then $g\circ f$ is a \vco map.
\item  Let $T'\subseteq \R^p$  be a \vco set. Then $f^{-1}(T')\subseteq
\R^n$ is a \vco set.
\end{itemize}
%-----------------end item------------------
\end{corollary}
%--- end-theorem-----------------------------------
%--- SUBsection{subsecStrat}-----------------------
\subsection{Stratifications and applications}
\label{subsecStrat}
%-----------------------------------------

We think that the results of this section could allow
to get most of the results
obtained by Frank Mausz in his Doctoral
dissertation \cite{Ma} with a different
approach.
%---- paragraph {paraSlSt}----

\ms\textbf{{\L}ojaziewicz stratification \`a la Cohen-H\"ormander}
%\label{paraSlSt}
%-----------------------------------------

We recall here a result about stratifying families
(\cite{BCR} chapter~9).
%-- Definota{defRootfunc}-----------
\begin{definota}
\label{defRootfunc}
Consider a general monic polynomial of degree $d$ as a point of
$\R^d$. Let
$\sigma=(\sigma_1,\ldots,\sigma_d)\in \{-1,+1\}^d$. Let
%-----------------begin $$----------------
$$U_\sigma=\left\{P\in\R^d\; :\; \exists x\in\R \;
\bigg( P(x)=0 \; \land\;  \Land\nolimits_{i=1}^d
\sign(P^{(i)}(x))=\sigma_i\bigg) \right\}.$$
%-----------------end $$------------------
It is easily seen that $U_\sigma$ is a connected open semialgebraic
subset of $\R^d$ (see e.g., \cite{GLM}) and that
%-----------------begin $$----------------
$$ \overline{U_\sigma}=\left\{P\in\R^d\; :\; \exists x\in\R \;
\bigg( P(x)=0 \; \land\;  \Land\nolimits_{i=1}^d
\sign(P^{(i)}(x))\in\{\sigma_i,0\}\bigg) \right\}.
$$
%-----------------end $$------------------
For $P\in U_\sigma$ we call $\rho_\sigma(P)$ the zero
which is coded \`a la
Thom by $(P,\sigma)$. Then $P\mapsto\rho_\sigma(P)$
is Nash on $U_\sigma$ and admits a continuous semialgebraic extension
on $\overline{U_\sigma}$, that we note also by
$\rho_\sigma$. Such a function will be called a {\em Thom's root
function}, or simply a {\em root function.}

More generally, if $\varphi:\R^{k-1}\rightarrow \R^d$ is a polynomial
function, we can consider $\rho_\sigma\circ \varphi $ as defined over
$\varphi^{-1}(\overline{U_\sigma})$. We also call such a function a
\emph{root function}. This function is Nash over
$\varphi^{-1}(U_\sigma)$.
If $f(x_1,\ldots,x_{k})=\varphi(x_1,\ldots,x_{k-1})(x_k)$ is the
corresponding monic polynomial in $k$ variables, we denote
  $\rho_\sigma\circ \varphi $ by $\rho_\sigma(f)$.

Finally if a polynomial
$g\in\K[x_1,\ldots,x_{k}]=\K[x_1,\ldots,x_{k-1}][x_k]$ has a
leading coefficient w.r.t.\  $x_k$ which is a nonzero element $c$ of
$\K$, we say that {\em $g$ is quasi monic in $x_k$}, and we let
  $\rho_\sigma(g)=\rho_\sigma(g/c).$
\end{definota}
%--- end-definota------------------------------------
For more details about root functions see \cite{GLM}.

%-- Theorem{thStra}-------------------
\begin{theorem} (\cite{BCR} chap. 9)
\label{thStra} Let $(\K,\P)$ be an ordered subfield of a \rcf
$(\R,\P_\R)$. Let $g_1,\ldots,g_s$ be nonzero polynomials in $\Kx$.
After a suitable linear change of variables
there exists a family of polynomials
$$(f_{i,j})_{i=1,\ldots,n;j=1,\ldots,\ell_i}$$ 
with the following
properties (we will continue denoting the new variables by $x_i$).
%-----------------begin item------------------
\begin{itemize}
\item [$(1)$] First we have
%-----------------begin item------------------
\begin{itemize}
\item  $(g_1,\ldots,g_s)\subseteq (f_{n,j})_{j=1,\ldots,\ell_n}$
\item Each $f_{k,j}$ is a nonzero polynomial
in $\K[x_1,\ldots,x_k]$ which is quasimonic in $x_k$.
\item  For each index $k$ the family
$(f_{k,j})_{j=1,\ldots,\ell_k}$ is stable under derivation w.r.t.\  $x_k$
(excluding the zero derivative).
\end{itemize}
%-----------------end item------------------
\item [$(2)$] Let us denote
$I_k=\{(i,j)\; :\; i=1,\ldots,k;j=1,\ldots,\ell_i \} $.
Call $\cC_k$ the family of nonempty semialgebraic
subsets of $\R^k$ that can be defined as some
$$
C_{\varepsilon}=\left\{(\xi_1,\ldots,\xi_k)\in\R^k\; ;\;
\Land\nolimits_{(i,j)\in I_k}\;
\sign(f_{i,j}(\xi_1,\ldots,\xi_k))=
\epsilon_{i,j} \right\}\neq \emptyset 
$$
(where $\varepsilon=(\epsilon_{i,j})_{(i,j)\in I_k}$
is any family in
$\left\{-1,0,+1\right\}$). It is clear that the $C_{\varepsilon}$'s in
$\cC_k$ give a partition of $\R^k$. We have
%-----------------begin item------------------
\begin{itemize}
\item [$(a)$] The canonical projection $\pi_k(C_{\varepsilon})$
of any element $C_{\varepsilon}\in\cC_k$ on $\R^{k-1}$ is an element
of $\cC_{k-1}$: it is obtained
as $C_{\varepsilon'}$ where $\varepsilon'$ is the restriction of
the family $\varepsilon$ to $I_{k-1}$.
\item [$(b)$] The adherence $\overline{C_{\varepsilon}}$ of
$C_{\varepsilon}$ (recall we assume $C_{\varepsilon}\neq \emptyset $)
is a union of elements of $\cC_k$, it is obtained by relaxing strict
inequalities in the definition of $C_{\varepsilon}$.
\item [$(c)$] If in the definition of $C_{\varepsilon}\in\cC_k$
there is one equality $f_{k,i}(\xi_1,\ldots,\xi_k)=0$
then $C_{\varepsilon}$ is the graph of a root function
$\rho_\sigma(f_{k,j})$  (here $f_{k,j}$ is seen as a polynomial
in $x_k$, it is equal to $f_{k,i}$ or to some
  $f_{k,i}^{(\ell)}$  and $\sigma$ is extracted from $\varepsilon$)
which is Nash over  $\pi_k(C_{\varepsilon})$. Moreover,
$\rho_\sigma(f_{k,i})$ is defined over
$\overline{\pi_k(C_{\varepsilon})}$
and the graph of this root function is $\overline{C_{\varepsilon}}$.
\item [$(d)$] Call $\pi_{n,k}$ the canonical projection
$\R^n\rightarrow \R^k$. Let $E$ be a $k$ dimensional semialgebraic
subset of $\R^n$ defined from the polynomials $g_1,\ldots,g_s$. Then
for any $C_{\varepsilon}\in\cC_n$ which is contained in $E$,
$\pi_{n,k}$ maps homeomorphically ${\overline{C_{\varepsilon}}}$  on
its image.
%\item [$(e)$]
\end{itemize}
%-----------------end item------------------
%\item [$(4)$]
\end{itemize}
%-----------------end item------------------
\end{theorem}
%--- end-theorem-----------------------------------------

%-- Definition{defThomStra}-----------
\begin{definition}
\label{defThomStra} Such a change of variables together with
such a family $(f_{i,j})
% _{i=1,\ldots,n;j=1,\ldots,\ell_i}
$  will be
called a {\em stratification} for
  $(g_1,\ldots,g_s)$ and for any semialgebraic subset of $\R^n$ defined
from this family. The family
$(f_{i,j})_{i=1,\ldots,n;j=1,\ldots,\ell_i}$  will be called a
{\em stratifying family} for the initial family
  $(g_1,\ldots,g_s)$. The semialgebraic subsets $C_{\varepsilon}$ are
called the {\em strata of the stratification}.
\end{definition}
%--- end-definition------------------------------------

We shall precisely consider the following way of constructing
a stratifying family,  \`a la Cohen-H\"ormander (it is the one
suggested in \cite{BCR}.)
First we make a linear change of variables in order to make
$g_1,\ldots,g_s$ quasi monic in the new variable $x_n$.
We add all the derivatives of each $g_i$ w.r.t.\  $x_n$. This gives us
the family $(f_{n,j})_{j=1,\ldots,\ell_n}$.

We apply Cohen-Hormander's algorithm to this family and we call
$h_1,\ldots,h_\ell$ the ``constants" given by this algorithm
(these constants are polynomials in $(x_1,\ldots,\alb x_{n-1})$).

We make a new linear change of variables
on $(x_1,\ldots,x_{n-1})$  in order to make
$h_1,\ldots,h_\ell$ quasi monic in the new variable $x_{n-1}$.
We make the same linear change of variables inside
$(f_{n,j})_{j=1,\ldots,\ell_n}$: this family remains quasimonic in
$x_n$ and stable under derivation w.r.t.\  $x_n$, and $h_1,\ldots,h_\ell$
remain the ``constants" given by the Cohen-Hormander's algorithm when
applied to this family.

We add all the derivatives of each $h_i$ w.r.t.\  $x_{n-1}$.
This gives
us the family $(f_{n-1,j})_{j=1,\ldots,\ell_{n-1}}$. And so on.

With this kind of stratifying family, we can apply recursively
Theorem \ref{propMCHV2}.
So we get a precise description of the variation
of the valuations $v(f_{k,j}(x_1,\ldots,x_k))$ when
  $(x_1,\ldots,x_k)\in C_\varepsilon$ for any $k$ and any
$C_\varepsilon\in\cC_k$.
Let us see an example.
%-- Example{exa5}----------------------
\begin{example}
\label{exa5}
Assume $n=3$. Consider a cell $C\in\cC_3$.
Assume that $C''=\pi_{3,1}(C)$ is an interval $\,]a,b[\,$, that
$C'=\pi_{3,2}(C)$ is the graph of a root function
$h_1=\rho_{\sigma}(f_{{2},1})$  defined on $[a,b]$,
and that $C$  is the part of
$C'\times \R$ between two root functions
$h_2=\rho_{\sigma'}(f_{{3},1})$  and
$h_3=\rho_{\sigma''}(f_{{3},2})$, so
%--------------------begin array---------------
$$\begin{array}{rcl}
C&=   & \left\{(x,y,z)\; :\; a<x<b,\; y=h_1(x),\;
h_2(x,y)<z<h_3(x,y) \right\} \\
& =  &  \left\{(x,y,z)\; :\; a<x<b,\; y=h_1(x),\;
h'_2(x)<z<h'_3(x) \right\}.
\end{array}$$
%---------------------end array--------------
We consider for $(x,y,z)\in C$, the parameters $t=(x-a)/(b-x)$,
$\tau=v(t)$, $t'=(z-h'_2(x))/(h'_3(x)-z)$ and
$\tau'=v(t')$. We get:
%-----------------begin item------------------
\begin{itemize}
\item The map  $h:(t,t')\mapsto (x,y,z)\in C$ is a
Nash isomorphism from $(\R^+)^2$ onto $C$.
\item For any $f_{k,j}$ in the stratifying family
$v(f_{k,j}(x,y,z))=\varphi_{k,j}(\tau,\tau')$ is a \qsl
function of $\tau,\tau'$
(here we use recursively Theorem \ref{propMCHV2}).
\item So, if we look at $C\, \cap\, S$ where $S$ is any
\vco subset described from the $f_{k,j}$'s, we find that
$C\, \cap\, S$ is a finite union of sets $h(L_i)$ where each
$L_i$ is defined as
%-----------------begin $$----------------
$$ \left\{(t,t')\in (\R^+)^2\; :\; \Land\nolimits_\ell
a_\ell(\tau,\tau')=\alpha_\ell
\; \land\; \Land\nolimits_m  b_m(\tau,\tau')>\beta_m\right\}
$$
%-----------------end $$------------------
where $a_\ell$'s and $b_m$'s are $\Z$-linear forms
and  $\alpha_\ell,\;\beta_m\in \GKd$.
\item Now we should like to have some rational expression of
$\tau$ and $\tau'$ that uses only polynomials in $(x,y,z)$. This is
possible in the following way, as in Remark \ref{rempropCHV2}.
Consider that the formal variables are
$X,Y,Z$ and that $x,y,z$ are three parameters. Add to the list $g_i$
the three polynomials $X-x,\; Y-y,\; Z-z$ and reconstruct the
stratification, using the information that $(x,y,z)$ is in the
semialgebraic set $C$. You get that \emph{$\tau$ and $\tau'$
are fixed \qsl functions in the $v(c_j)$'s and in some $v(F_j(x,y,z))$'s}:
the $c_j$'s are the old constants, and the
$F_j(x,y,z)$ are the new ``constants" that are constructed by the
algorithm ($F_j(x,y,z)\in\K[x,y,z]$).
\end{itemize}
%-----------------end item------------------
\end{example}
%--- end-example-----------------------------------------

%-- Theorem{thCellDec1}---------------

The following ``cell decomposition theorem"
is merely the generalization of what we
have seen on this example. It is obtained by applying Theorem
\ref{thValThom} to a stratification \`a la Cohen-H\"ormander.
The last assertion is obtained as in Remark \ref{rempropCHV2}.

\begin{theorem}
\label{thCellDec1} {\em (Cell decomposition theorem)}
Let $\KVP$ be an \ovsf of a \rcvf $\RVP$.
  Let $g_1,\ldots,g_s$ be nonzero polynomials in $\Kx$.
Consider a linear change of variables together with
  a family $(f_{i,j})_{i=1,\ldots,n;j=1,\ldots,\ell_i}$  that give a
stratification for $(g_1,\ldots,g_s)$. Assume that this stratification
is constructed \`a la Cohen-H\"ormander, as explained above (after
Definition \ref{defThomStra}). Consider any $k$-dimensional stratum
$C_\varepsilon$
corresponding to this stratification (see Theorem \ref{thStra}).
Then there is a Nash isomorphism
$$h:(\R^+)^k\longrightarrow C_\varepsilon,\; \;\;
(t_1,\ldots,t_k)\longmapsto h(t_1,\ldots,t_k)$$
with the following property.

If  $S$ is any \vco subset described from $g_1,\ldots,g_s$, then
  $S\, \cap\, C_\varepsilon$ is a finite union of cells $h(L_i)$, where each
$L_i$  can be defined as
%-----------------begin $$----------------
$$ \left\{(t_1,\ldots,t_k)\in (\R^+)^k\; :\; \Land\nolimits_\ell
a_\ell(\tau)=\alpha_\ell
\; \land\; \Land\nolimits_m  b_m(\tau)>\beta_m\right\}
$$
%-----------------end $$------------------
where $\tau=(\tau_1,\ldots,\tau_k)=(v(t_1),\ldots,v(t_k)$,  the $a_\ell$'s and
$b_m$'s
are $\Z$-linear forms w.r.t.\
$\tau$,
and  $\alpha_\ell,\;\beta_m\in \GKd$.

Moreover, each $\tau_i$ is a \qsl function in some
$v(F_j(x_1,\ldots,x_n))$'s (with $F_j$'s explicitly computable elements of
$\Kx$).
\end{theorem}
%--- end-theorem-----------------------------------------

\refstepcounter{bidon}
\addcontentsline{toc}{section}{References}

%------ BIBLIO ----------------------------------------

%-------------------------------------------------

\tableofcontents

\end{document}